\documentclass[11pt,reqno,letterpaper]{amsart}

\usepackage{amssymb,amsmath,amsthm,amsfonts}
\usepackage{bbm,enumerate}

\usepackage{bookmark}
\usepackage{hyperref}
\hypersetup{pdfstartview={FitH}}

\theoremstyle{plain}
\newtheorem{theorem}{Theorem}
\newtheorem{lemma}[theorem]{Lemma}
\newtheorem{proposition}[theorem]{Proposition}

\theoremstyle{definition}
\newtheorem{example}{Example}

\theoremstyle{remark}
\newtheorem{remark}{Remark}

\numberwithin{equation}{section}

\renewcommand{\leq}{\leqslant}
\renewcommand{\geq}{\geqslant}

\begin{document}

\title[Bilinear embedding in Orlicz spaces]{Bilinear embedding in Orlicz spaces for divergence-form operators with complex coefficients}

\author[V. Kova\v{c}]{Vjekoslav Kova\v{c}}
\address{Vjekoslav Kova\v{c}, Department of Mathematics, Faculty of Science, University of Zagreb, Bijeni\v{c}ka cesta 30, 10000 Zagreb, Croatia}
\email{vjekovac@math.hr}

\author[K. A. \v{S}kreb]{Kristina Ana \v{S}kreb}
\address{Kristina Ana \v{S}kreb, Faculty of Civil Engineering, University of Zagreb, Fra Andrije Ka\v{c}i\'{c}a Mio\v{s}i\'{c}a 26, 10000 Zagreb, Croatia}
\email{kskreb@grad.hr}



\begin{abstract}
We prove a bi-sublinear embedding for semigroups generated by non-smooth complex-coefficient elliptic operators in divergence form and for certain mutually dual pairs of Orlicz-space norms. This generalizes a result by Carbonaro and Dragi\v{c}evi\'{c} from power functions to more general Young functions that still behave like powers. To achieve this, we generalize a Bellman function constructed by Nazarov and Treil.
\end{abstract}

\maketitle


\section{Introduction}
One is often lead to study bi-sublinear estimates of the form
\begin{equation}\label{eq:oldbilinearemb}
\int_{0}^{\infty} \int_{\mathbb{R}^d} \big|\nabla T_t f(x)\big| \,\big|\nabla \widetilde{T}_t g(x)\big| \,\textup{d}x \,\textup{d}t \leqslant C \|f\| \|g\|^{\ast},
\end{equation}
where $f,g$ are complex functions, $\|\cdot\|$, $\|\cdot\|^{\ast}$ are mutually dual Banach space norms, and $(T_t)_{t>0}$, $(\widetilde{T}_t)_{t>0}$ are operator semigroups. Here, $|\cdot|$ simply denotes the standard (i.e., Euclidean) norm on $\mathbb{C}^d$ and we emphasize that the constant $C$ depends on the norms and the semigroups, but not on the functions. Inequalities \eqref{eq:oldbilinearemb} are often called \emph{bilinear embeddings} (even though they are only bi-sublinear) and they are highly prized in the literature.
As early examples, Petermichl and Volberg \cite{PetermichlVolberg02} and Nazarov and Volberg \cite{NazarovVolberg03} studied such embeddings in the context of bounds for the Ahlfors--Beurling operator. Dragi{\v{c}}evi\'{c} and Volberg \cite{DragicevicVolberg05,DragicevicVolberg06,DragicevicVolberg11,DragicevicVolberg12} established a series of dimension-free estimates of type \eqref{eq:oldbilinearemb} in versatile analytical contexts, for both classical and fairly general semigroups. More recently, Carbonaro and Dragi\v{c}evi\'{c} proved several bilinear embeddings of type \eqref{eq:oldbilinearemb} and used them to study bounds for the Riesz transforms associated with Riemannian manifolds \cite{CarbonaroDragicevic13}, extend the functional calculus for generators of symmetric contraction semigroups \cite{CarbonaroDragicevic17}, and shed a new light on properties of semigroups associated with divergence-form operators with complex coefficients \cite{CarbonaroDragicevic20}.
We have not attempted to list all existing literature as bilinear embeddings are the topic of much recent and ongoing research.

In this paper we study bilinear embeddings \eqref{eq:oldbilinearemb} on the Orlicz function spaces for semigroups generated by elliptic operators with bounded measurable complex coefficients; see Subsection~\ref{subsec:formulationresults} for precise formulation of the result. In particular, we reprove and generalize the main result from \cite{CarbonaroDragicevic20}, which was concerned with $\textup{L}^p$ norms only.
In more detail, but still briefly, the following notions characterize our setting and approach.

\begin{enumerate}[(i)]

\item\label{item:setting1}
\emph{Non-smooth complex divergence-form operators} will be discussed in Subsection~\ref{subsec:divform}. Numerous results that hold for real divergence-form operators generally fail for their complex counterparts. Thus, it is an active line of research to give sufficient conditions for the corresponding estimates in the complex case; see \cite{CialdeaMazya05,CarbonaroDragicevic13, CarbonaroDragicevic17,CarbonaroDragicevic19,CarbonaroDragicevic20,CarbonaroDragicevic20b,CialdeaMazya21,CarbonaroDragicevic19b,CarbonaroDragicevicKovacSkreb21,CialdeaMazya21b}.
The notion of $p$-ellipticity, introduced by Carbonaro and Dragi\v{c}evi\'{c} in \cite{CarbonaroDragicevic20} and reviewed in our Subsection~\ref{subsec:divform}, proved to be useful in relation with $\textup{L}^p$ estimates, as it provides a gradation of assumptions stretched between real ellipticity and (complex) ellipticity.
Interesting aspects of the theory also happen on domains $\Omega\subseteq\mathbb{R}^d$ (see \cite{CarbonaroDragicevic20b,CarbonaroDragicevic19b,CarbonaroDragicevicKovacSkreb21}), but here we choose to work exclusively on $\mathbb{R}^d$, which avoids numerous technical complications.

\item
In all of the aforementioned papers, the norms $\|\cdot\|$ and $\|\cdot\|^{\ast}$ are just the (unweighted or weighted) $\textup{L}^p$ norms.
\emph{Young functions} and \emph{Orlicz spaces} will be reviewed in Subsection~\ref{subsec:Orlicz}. One way of thinking about those spaces is as both providing a refinement of the scale of $\textup{L}^p$ spaces and offering substitutes for the missing endpoint estimates.
For these reasons, the Orlicz norms frequently appear in harmonic analysis, but it seems that, so far, no Orlicz-space estimates have been studied in the context of \eqref{eq:oldbilinearemb} and semigroups generated by operators from \eqref{item:setting1}.
Related ``functional'' estimates for complex divergence-form operators have recently been discussed by Cialdea and Maz'ya \cite{CialdeaMazya21,CialdeaMazya21b}, in the context of certain generalized dissipativity of operators from \eqref{item:setting1}, and we also find \cite{CialdeaMazya21} motivating for the setting of the present paper.

\item
Our results will be established via the \emph{heat flow method}, a particular case of the \emph{Bellman function technique}. This is certainly not surprising, as the proofs of all aforementioned $\textup{L}^p$ bilinear embeddings proceeded precisely this way. In fact, we will closely follow the basic outline from \cite{CarbonaroDragicevic20}. However, each of the papers by Carbonaro and Dragi\v{c}evi\'{c} \cite{CarbonaroDragicevic13, CarbonaroDragicevic17,CarbonaroDragicevic19,CarbonaroDragicevic20,CarbonaroDragicevic20b,CarbonaroDragicevic19b} used (slight variants of) the Bellman function constructed by Nazarov and Treil \cite{NazarovTreil96}, while here we need to construct a Bellman function tailored to a pair of complementary Young functions (see the definition in Subsection~\ref{subsec:Orlicz}), which generalizes the Nazarov--Treil Bellman function. We provide one such function in Section~\ref{sec:Bellmanfn}. This might be an interesting result on its own, as most of this paper is dedicated to verification of the numerous required properties of the constructed function, such as the \emph{generalized convexity} introduced in \cite{CarbonaroDragicevic20} and discussed in our Lemma~\ref{lm:Xlower} below. We also believe that this construction could find further applications in loosely related contexts.
Very few papers construct Bellman functions to prove Orlicz-space estimates on $\mathbb{R}^d$; see \cite{TreilVolberg16} for an example.

\end{enumerate}

The study of generalized convexity for more general Young functions, in connection with Bellman functions and Orlicz-space estimates, was suggested by Alexander Volberg in the summer of 2016; this has been communicated to us by Oliver Dragi\v{c}evi\'{c}.

Structure of the present paper is as follows.
Section~\ref{sec:formulation} recalls the basic definitions and clarifies the lengthy assumptions needed later, regarding both the Young functions (Subsection~\ref{subsec:Orlicz}) and divergence-form operators (Subsection~\ref{subsec:divform}). Then it proceeds with formulation of the main result, namely Theorem~\ref{thm:mainthm}, and gives numerous remarks on its applicability (Subsection~\ref{subsec:formulationresults}).
Section~\ref{sec:Hessians} recalls the concept of a generalized Hessian from \cite{CarbonaroDragicevic20} and computes two expressions associated with rather general nonlinear functions.
Section~\ref{sec:Bellmanfn} is the heart of the paper. It constructs the Bellman function \eqref{eq:mainBellman} corresponding to the studied problem and proves a series of its delicate properties needed in the proof of the main theorem.
Section~\ref{sec:proofofthm} completes the proof of Theorem~\ref{thm:mainthm} by closely following the scheme from \cite{CarbonaroDragicevic20}.

\section{Formulation of the main result}
\label{sec:formulation}
\subsection{Young functions and Orlicz spaces}
\label{subsec:Orlicz}
We only review the basic definitions; more details can be found in the books \cite{RaoRen91,HarjulehtoHasto19}.
Let $\Phi\colon[0,\infty)\to[0,\infty)$ be a \emph{Young function}, i.e.,
\begin{equation}\label{eq:defYoung}
\Phi\text{ is convex},\quad \Phi(0)=0,\quad \lim_{s\to0+}\frac{\Phi(s)}{s}=0,\quad\text{and } \lim_{s\to\infty}\frac{\Phi(s)}{s}=\infty.
\end{equation}
Let $\Phi_{\ast}\colon[0,\infty)\to[0,\infty)$ be the \emph{complementary (or conjugate) Young function} to $\Phi$, defined as
\begin{equation}\label{eq:complementary}
\Phi_{\ast}(t) := \sup_{s\in(0,\infty)} \big(st - \Phi(s)\big)  = \int_{0}^{t} (\Phi')^{-1}(r) \,\textup{d}r,
\end{equation}
where the integral expression for $\Phi_{\ast}$ can be used in special cases when $(\Phi')^{-1}$ is well-defined on $(0,\infty)$.
This definition ensures that \emph{Young's inequality} holds:
\begin{equation}\label{eq:Youngsineq}
s t \leqslant \Phi(s) + \Phi_{\ast}(t) \quad\text{for } s,t\in[0,\infty).
\end{equation}
The Orlicz-space \emph{Luxemburg norm} $\|\cdot\|_\Phi$ is defined for (classes of a.e.\@ equal) measurable complex functions $f$ on $\mathbb{R}^d$ as
\begin{equation}\label{eq:Luxnorms}
\|f\|_\Phi := \inf\Big\{ \alpha\in(0,\infty) : \int_{\mathbb{R}^d} \Phi\Big(\frac{|f(x)|}{\alpha}\Big) \,\textup{d}x \leqslant 1 \Big\}.
\end{equation}

\begin{remark}
\label{rem:dualnorms}
Note that we will be working simultaneously with two norms, $\|\cdot\|_\Phi$ and $\|\cdot\|_{\Phi_{\ast}}$. These are said to be \emph{complementary} or \emph{mutually associate}, but they do not need to be mutually dual. In order for $\|\cdot\|_{\Phi_{\ast}}$ to be equivalent to the dual of $\|\cdot\|_\Phi$ it is sufficient that $\Phi$ is \emph{doubling}, i.e., there exists a constant $K$ such that
\[ \Phi(2s) \leqslant K \Phi(s) \quad\text{for } s\in[0,\infty). \]
Indeed, doubling functions $\Phi$ are said to ``satisfy globally the $\Delta_2$-condition'' in \cite[Section~2.3, Definition~1]{RaoRen91}, so the reader can deduce the last claim from \cite[Section~3.4, Corollary~5]{RaoRen91} and \cite[Section~4.1, Theorem~6]{RaoRen91}.
\end{remark}

We will need to narrow down the above setting in order to obtain meaningful results. Throughout the paper we assume the following:
{\allowdisplaybreaks
\begin{subequations}
\begin{align}
& \text{$\Phi$ and ${\Phi_{\ast}}$ are mutually complementary Young functions}, \label{eq:PhiPsicond1} \\
& \Phi \text{ and } {\Phi_{\ast}} \text{ are } \textup{C}^1 \text{ on } [0,\infty) \text{ and } \textup{C}^2 \text{ on } (0,\infty), \label{eq:PhiPsicond2} \\
& \Phi''(s),\Phi_{\ast}''(s)>0 \text{ for } s\in(0,\infty), \label{eq:PhiPsicond3} \\
& \Phi' \text{ is strictly convex on } (0,\infty) \text{ and } \lim_{s\to0^+}\frac{\Phi'(s)}{s}=0, \label{eq:Phicond1} \\
& \sup_{s\in(0,\infty)}\frac{s\Phi'(s)}{\Phi(s)} < \infty, \label{eq:Phicond2} \\
& 1 < \inf_{s\in(0,\infty)}\frac{s\Phi''(s)}{\Phi'(s)} \leqslant \sup_{s\in(0,\infty)}\frac{s\Phi''(s)}{\Phi'(s)} < \infty. \label{eq:Phicond3}
\end{align}
\end{subequations}
}
Note that defining properties \eqref{eq:defYoung} and assumptions \eqref{eq:PhiPsicond1}--\eqref{eq:PhiPsicond3} imply that
\begin{equation}\label{eq:bijections}
\text{$\Phi'$ and $\Phi_{\ast}'$ are mutually inverse increasing bijections of $[0,\infty)$}.
\end{equation}
Because of that, assuming \eqref{eq:PhiPsicond1}--\eqref{eq:PhiPsicond3}, conditions \eqref{eq:Phicond1}--\eqref{eq:Phicond3} are respectively equivalent to conditions:
{\allowdisplaybreaks
\begin{subequations}
\begin{align}
& \Phi_{\ast}' \text{ is strictly concave on } (0,\infty) \text{ and } \lim_{s\to0^+}\frac{\Phi_{\ast}'(s)}{s}=\infty, \label{eq:Psicond1} \\
& \inf_{s\in(0,\infty)}\frac{s\Phi_{\ast}'(s)}{{\Phi_{\ast}}(s)} > 1, \label{eq:Psicond2} \\
&  0 < \inf_{s\in(0,\infty)}\frac{s\Phi_{\ast}''(s)}{\Phi_{\ast}'(s)} \leqslant \sup_{s\in(0,\infty)}\frac{s\Phi_{\ast}''(s)}{\Phi_{\ast}'(s)} < 1. \label{eq:Psicond3}
\end{align}
\end{subequations}
}
Indeed, equivalence \eqref{eq:Phicond1}$\Longleftrightarrow$\eqref{eq:Psicond1} is an immediate consequence of \eqref{eq:bijections} and
\[ \lim_{t\to0^+} \frac{t}{\Phi_{\ast}'(t)}
= \big[\text{substitute }\ t=\Phi'(s) \,\Longleftrightarrow\, s=\Phi_{\ast}'(t)\big]
= \lim_{s\to0^+}\frac{\Phi'(s)}{s}. \]
Moreover, equivalence \eqref{eq:Phicond3}$\Longleftrightarrow$\eqref{eq:Psicond3} clearly follows from
\begin{align}
\Big\{\frac{\Phi_{\ast}'(t)}{t\Phi_{\ast}''(t)} : t\in(0,\infty)\Big\}
& = \Big\{\frac{\Phi_{\ast}'(t)\Phi''(\Phi_{\ast}'(t))}{t} : t\in(0,\infty)\Big\} \nonumber \\
& \quad \big[\text{substitute }\ t=\Phi'(s) \,\Longleftrightarrow\, s=\Phi_{\ast}'(t)\big] \nonumber \\
& = \Big\{\frac{s\Phi''(s)}{\Phi'(s)} : s\in(0,\infty)\Big\}. \label{eq:weicrucomp}
\end{align}
Finally, computation
{\allowdisplaybreaks
\begin{align*}
\Big\{ \frac{\Phi(s)}{s\Phi'(s)} : s\in(0,\infty)\Big\}
& = \Big\{ \frac{1}{s\Phi'(s)} \int_{0}^{s} \Phi'(u) \,\textup{d}u : s\in(0,\infty)\Big\} \\
& \quad \big[\text{substitute }\ s=\Phi_{\ast}'(t) \,\Longleftrightarrow\, t=\Phi'(s)\big] \\
& = \Big\{ \frac{1}{\Phi_{\ast}'(t)t} \int_{0}^{\Phi_{\ast}'(t)} \Phi'(u) \,\textup{d}u : t\in(0,\infty)\Big\} \\
& \quad \Big[\begin{array}{c}u=\Phi_{\ast}'(v)\\ \textup{d}u=\Phi_{\ast}''(v)\,\textup{d}v\end{array}\Big] \\
& = \Big\{ \frac{1}{t\Phi_{\ast}'(t)} \int_{0}^{t} v \Phi_{\ast}''(v) \,\textup{d}v : t\in(0,\infty)\Big\} \\
& \quad \big[ \text{integration by parts} \big] \\
& = \Big\{ 1 - \frac{{\Phi_{\ast}}(t)}{t\Phi_{\ast}'(t)} : t\in(0,\infty)\Big\}
\end{align*}
}
shows \eqref{eq:Phicond2}$\Longleftrightarrow$\eqref{eq:Psicond2}.

It has already been implied in \eqref{eq:Phicond2} and \eqref{eq:Phicond3} that the following four quantities will be relevant later. They can be defined in terms of $\Phi$ as
\begin{align}
m := \inf_{s\in(0,\infty)}\frac{s\Phi'(s)}{\Phi(s)}, & \quad
M := \sup_{s\in(0,\infty)}\frac{s\Phi'(s)}{\Phi(s)}, \label{eq:quantities1} \\
\tilde{m} := \inf_{s\in(0,\infty)}\frac{s\Phi''(s)}{\Phi'(s)}, & \quad
\tilde{M} := \sup_{s\in(0,\infty)}\frac{s\Phi''(s)}{\Phi'(s)}, \label{eq:quantities2}
\end{align}
or, equivalently, thanks to the previous computations, in terms of ${\Phi_{\ast}}$ via
\begin{align}
\frac{M}{M-1} = \inf_{s\in(0,\infty)}\frac{s\Phi_{\ast}'(s)}{{\Phi_{\ast}}(s)}, & \quad
\frac{m}{m-1} = \sup_{s\in(0,\infty)}\frac{s\Phi_{\ast}'(s)}{{\Phi_{\ast}}(s)}, \label{eq:quantities3} \\
\frac{1}{\tilde{M}} = \inf_{s\in(0,\infty)}\frac{s\Phi_{\ast}''(s)}{\Phi_{\ast}'(s)}, & \quad
\frac{1}{\tilde{m}} = \sup_{s\in(0,\infty)}\frac{s\Phi_{\ast}''(s)}{\Phi_{\ast}'(s)}. \label{eq:quantities4}
\end{align}
Conditions \eqref{eq:PhiPsicond2}--\eqref{eq:Phicond1} imply that $\Phi''$ is continuous and increasing, so for any $s\in(0,\infty)$ Chebyshev's rearrangement inequality (see \cite[Section~2.17, Theorem~43]{HardyLittlewoodPolya52} and \cite[Chapter~6, Theorem~236]{HardyLittlewoodPolya52}) gives
\[ \frac{\Phi(s)}{s} = \frac{1}{s} \int_0^s (s-u) \Phi''(u) \,\textup{d}u
\leqslant \Big( \frac{1}{s} \int_0^s (s-u) \,\textup{d}u \Big) \Big( \frac{1}{s} \int_0^s \Phi''(u) \,\textup{d}u \Big) = \frac{1}{2}\Phi'(s). \]
Thus, our assumptions \eqref{eq:PhiPsicond1}--\eqref{eq:Phicond3} guarantee
\begin{equation}\label{eq:onmMs}
2\leqslant m\leqslant M<\infty,\quad 1<\tilde{m}\leqslant \tilde{M}<\infty.
\end{equation}
Consequences of \eqref{eq:Phicond1}, \eqref{eq:Phicond3}, \eqref{eq:Psicond1}, \eqref{eq:Psicond3}, \eqref{eq:quantities2}, and \eqref{eq:quantities4} are
\begin{align*}
(\tilde{m}-1)\frac{\Phi'(s)}{s^2} \leqslant & \frac{\textup{d}}{\textup{d}s}\frac{\Phi'(s)}{s} \leqslant (\tilde{M}-1)\frac{\Phi'(s)}{s^2}, \\
\frac{1-1/\tilde{m}}{\Phi_{\ast}'(s)} \leqslant & \frac{\textup{d}}{\textup{d}s}\frac{s}{\Phi_{\ast}'(s)} \leqslant \frac{1-1/\tilde{M}}{\Phi_{\ast}'(s)},
\end{align*}
and thus, by integrating in $s$, also
\begin{align}
\frac{1}{\tilde{M}-1} \frac{\Phi'(t)}{t} \leqslant & \int_{0}^{t}\frac{\Phi'(s)\,\textup{d}s}{s^2} \leqslant \frac{1}{\tilde{m}-1} \frac{\Phi'(t)}{t}, \label{eq:intPhiupper} \\
\frac{\tilde{M}}{\tilde{M}-1} \frac{t}{\Phi_{\ast}'(t)} \leqslant & \int_{0}^{t}\frac{\textup{d}s}{\Phi_{\ast}'(s)} \leqslant \frac{\tilde{m}}{\tilde{m}-1} \frac{t}{\Phi_{\ast}'(t)} \label{eq:intPsiupper}
\end{align}
for every $t\in(0,\infty)$.

Let us also remark that $\Phi$ and ${\Phi_{\ast}}$ satisfying \eqref{eq:PhiPsicond1}--\eqref{eq:Phicond3}, and thus also \eqref{eq:Psicond1}--\eqref{eq:Psicond3}, will automatically be doubling, as defined in Remark~\ref{rem:dualnorms}. This is easily seen as
\begin{align*}
\frac{\Phi(2s)}{\Phi(s)} & = \exp\Big(\int_{1}^{2}\frac{st\Phi'(st)}{\Phi(st)}\frac{\textup{dt}}{t}\Big) \leqslant 2^M, \\
\frac{{\Phi_{\ast}}(2s)}{{\Phi_{\ast}}(s)} & = \exp\Big(\int_{1}^{2}\frac{st\Phi_{\ast}'(st)}{{\Phi_{\ast}}(st)}\frac{\textup{dt}}{t}\Big) \leqslant 2^{m/(m-1)} \leqslant 4
\end{align*}
for $s\in(0,\infty)$.
The first inequality above used \eqref{eq:quantities1}, while the second one used \eqref{eq:quantities3} and \eqref{eq:onmMs}.
Consequently, $\|\cdot\|_{\Phi_{\ast}}\sim\|\cdot\|_\Phi^\ast$ and $\|\cdot\|_\Phi\sim\|\cdot\|_{\Phi_{\ast}}^\ast$, where $\|\cdot\|^\ast$ denotes the dual norm of $\|\cdot\|$ and $\sim$ denotes the equivalence of norms.

\begin{example}[Lebesgue spaces $\textup{L}^p$]
\label{ex:Lpnorms}
A typical example of a pair of functions $\Phi,{\Phi_{\ast}}$ for which the above conditions \eqref{eq:PhiPsicond1}--\eqref{eq:Phicond3} hold is
\begin{equation}\label{eq:LpLq}
\Phi(s)=\frac{s^p}{p},\quad {\Phi_{\ast}}(s)=\frac{s^q}{q}\quad \text{for } p\in(2,\infty),\ q\in(1,2),\ \frac{1}{p}+\frac{1}{q}=1.
\end{equation}
In this case $\|\cdot\|_\Phi\sim\|\cdot\|_{\textup{L}^p}$ and $\|\cdot\|_{\Phi_{\ast}}\sim\|\cdot\|_{\textup{L}^q}$.
Also note that
\[ m=M=p,\quad \tilde{m}=\tilde{M}=p-1. \]
\end{example}

\begin{example}[Zygmund spaces $\textup{L}^r\log\textup{L}$]
\label{ex:LrlogL}
Conditions \eqref{eq:PhiPsicond1}--\eqref{eq:Phicond3} are also satisfied for functions that ``behave like powers.''
We can take
\[ \Phi(s) = s^r \log(s+e) \quad \text{for } r\in(2,\infty), \]
while we cannot, and do not need to, evaluate its conjugate function ${\Phi_{\ast}}$ explicitly.
Exact expressions for $M$ and $\tilde{M}$ involve a bit complicated numerical constants, but we always have
\[ r=m\leqslant M< r+1,\quad r-1=\tilde{m}\leqslant \tilde{M}< r. \]
\end{example}

\begin{example}[Superpositions of powers I]
\label{ex:powers1}
Yet another useful example satisfying \eqref{eq:PhiPsicond1}--\eqref{eq:Phicond3} is
\begin{equation}\label{eq:spsr}
\Phi(s) = s^p + \varepsilon s^r \quad \text{for } 2<r<p<\infty,\ \varepsilon\in(0,1].
\end{equation}
This Young function exhibits the features of $s^r$ for small positive  $s$ and those of $s^p$ for large $s$.
We have
\[ m=r,\quad M=p,\quad \tilde{m}=r-1,\quad \tilde{M}=p-1 \]
and note that these quantities are independent of $\varepsilon$.
A straightforward generalization of this example is
\[ \Phi(s) = \int s^t \,\textup{d}\mu(t) \]
for a finite positive Borel measure $\mu$ supported on a compact subinterval of $(2,\infty)$.
\end{example}

\begin{example}[Superpositions of powers II]
\label{ex:powers2}
Define $\Phi$ as the conjugate function of
\[ {\Phi_{\ast}}(s) = s^q + s^r \quad \text{for } 1<q<r<2. \]
This time $\Phi$ cannot, and does not need to, be evaluated explicitly.
It is now more convenient to verify conditions \eqref{eq:bijections}, \eqref{eq:PhiPsicond2}--\eqref{eq:PhiPsicond3}, and \eqref{eq:Psicond1}--\eqref{eq:Psicond3}, which are sufficient by the previous discussion. Moreover, the four characteristic quantities can be computed from \eqref{eq:quantities3} and \eqref{eq:quantities4}, and they equal
\[ m=\frac{r}{r-1},\quad M=\frac{q}{q-1},\quad \tilde{m}=\frac{r}{r-1}-1,\quad \tilde{M}=\frac{q}{q-1}-1. \]
A generalization of this example is
\[ {\Phi_{\ast}}(s) = \int s^t \,\textup{d}\mu(t), \]
where $\mu$ is a finite positive Borel measure supported on a compact subinterval of $(1,2)$.
\end{example}

\subsection{Divergence-form operators with non-smooth complex coefficients}
\label{subsec:divform}
Once again, we only give the basic definitions; more details can be found in the book by Ouhabaz \cite{Ouhabaz05}.
Let $A\colon\mathbb{R}^d\to\mathbb{C}^{d\times d}$ be a matrix function with coefficients in $\textup{L}^\infty(\mathbb{R}^d)$.
It is said to be \emph{(uniformly) elliptic} if
\begin{align}
\Lambda(A) & := \mathop{\textup{ess\,sup}}_{x\in\mathbb{R}^d} \max_{\substack{\zeta,\eta\in\mathbb{C}^d\\ |\zeta|=|\eta|=1}} \big|\langle A(x)\zeta,\eta\rangle_{\mathbb{C}^d}\big| < \infty, \label{eq:condbigl} \\
\lambda(A) & := \mathop{\textup{ess\,inf}}_{x\in\mathbb{R}^d} \min_{\substack{\xi\in\mathbb{C}^d\\ |\xi|=1}} \mathop{\textup{Re}}\big\langle A(x)\xi,\xi\big\rangle_{\mathbb{C}^d} > 0. \label{eq:condlittlel}
\end{align}
Define the corresponding \emph{divergence-form operator} formally as
\[ L_A f := -\mathop{\textup{div}}(A\nabla f). \]
More precisely, $L_A$ is defined via duality:
\begin{equation}\label{eq:sesqform}
\langle L_A f, g \rangle_{\textup{L}^2(\mathbb{R}^d)} = \int_{\mathbb{R}^d} \langle A(x)\nabla f(x), \nabla g(x) \rangle_{\mathbb{C}^d} \,\textup{d}x
\end{equation}
and its domain $\mathcal{D}(L_A)$ is the set of all functions $f$ from the Sobolev space $\textup{W}^{1,2}(\mathbb{R}^d)$ for which the right hand side of \eqref{eq:sesqform}, regarded as an antilinear functional in $g\in\textup{W}^{1,2}(\mathbb{R}^d)$, extends boundedly to the whole $\textup{L}^2(\mathbb{R}^d)$.
We will consider the operator semigroup on $\textup{L}^2(\mathbb{R}^d)$ generated by $-L_A$:
\[ T^A_t := \exp(-t L_A) \quad\text{for } t\in(0,\infty). \]

Carbonaro and Dragi\v{c}evi\'{c} \cite{CarbonaroDragicevic20} introduced the property of \emph{$p$-ellipticity} of $A$ for $p\in[1,\infty]$ by additionally requiring $\Delta_{p}(A)>0$, where
\begin{align}
\Delta_{p}(A) & := \mathop{\textup{ess\,inf}}_{x\in\mathbb{R}^d} \min_{\substack{\xi\in\mathbb{C}^d\\ |\xi|=1}} \mathop{\textup{Re}}\Big\langle A(x)\xi, \ \xi + \Big|1-\frac{2}{p}\Big| \,\overline{\xi} \Big\rangle_{\mathbb{C}^d} \nonumber \\
& = \mathop{\textup{ess\,inf}}_{x\in\mathbb{R}^d} \min_{\substack{\xi\in\mathbb{C}^d\\ |\xi|=1}} \mathop{\textup{Re}}\Big\langle A(x)\xi, \ \xi + \Big(1-\frac{2}{p}\Big) \,\overline{\xi} \Big\rangle_{\mathbb{C}^d}. \label{eq:condpellip}
\end{align}
An equivalent condition was discovered independently by Dindo\v{s} and Pipher \cite{DindosPipher19} as a strengthening of the earlier condition introduced by Cialdea and Maz'ya \cite{CialdeaMazya05}.
It is also easy to check that for $2\leqslant p_1\leqslant p_2<\infty$ we have
\[ \lambda(A) = \Delta_2(A) \geqslant \Delta_{p_1}(A) \geqslant \Delta_{p_2}(A) \]
and that the following inclusions hold:
\begin{align}
& {\arraycolsep=2pt
\left\{\begin{array}{c}\text{elliptic}\\ \text{matrices}\end{array}\right\}
= \left\{\begin{array}{c}\text{$2$-elliptic}\\ \text{matrices}\end{array}\right\}
\supseteq \left\{\begin{array}{c}\text{$p_1$-elliptic}\\ \text{matrices}\end{array}\right\}
} \nonumber \\
& {\arraycolsep=2pt
\supseteq \left\{\begin{array}{c}\text{$p_2$-elliptic}\\ \text{matrices}\end{array}\right\}
\supseteq \left\{\begin{array}{c}\text{matrices that are $p$-elliptic}\\ \text{for every }p\in[2,\infty)\end{array}\right\}
= \left\{\begin{array}{c}\text{real elliptic}\\ \text{matrices}\end{array}\right\};
} \label{eq:pellinclusions}
\end{align}
see \cite[Section~5.3]{CarbonaroDragicevic20}.
Therefore, the notion of $p$-ellipticity bridges the gap between real and complex elliptic matrix functions.

\smallskip
Let us now take a complex matrix function $A$ and an arbitrary $\textup{C}^2$ function $\Psi\colon(0,\infty)\to(0,\infty)$ such that $\Psi'(s)>0$ and $\Psi''(s)>0$ for every $s\in(0,\infty)$.
We define a quantity $\Delta_{\Psi}(A)$ ``measuring'' certain $\Psi$-ellipticity of $A$ as
\begin{equation}\label{eq:condPhi1}
\Delta_{\Psi}(A) := \mathop{\textup{ess\,inf}}_{x\in\mathbb{R}^d} \inf_{\substack{\xi\in\mathbb{C}^d,\,|\xi|=1\\ s\in(0,\infty)}} \mathop{\textup{Re}}\Big\langle A(x)\xi, \ \xi + \frac{s\Psi''(s)-\Psi'(s)}{s\Psi''(s)+\Psi'(s)} \,\overline{\xi} \Big\rangle_{\mathbb{C}^d}
\end{equation}
and say that $A$ is \emph{$\Psi$-elliptic} if $\Delta_{\Psi}(A)>0$.
Indeed, \eqref{eq:condPhi1} reduces to \eqref{eq:condpellip} when $\Psi(s)=s^p/p$.
However, one could merely generalize \eqref{eq:condpellip} in many other ways, so let us spend a few more words motivating the above definition and explaining why it is more useful than the other possibilities.

\begin{remark}
One way to motivate the definition \eqref{eq:condpellip} of $\Delta_p(A)$ was to differentiate formally using \eqref{eq:sesqform}:
\begin{align*}
\frac{\textup{d}}{\textup{d}t}\Big|_{t=0} \big\| (T^A_t f)(x) \big\|_{\textup{L}^{p}(\mathbb{R}^d)}^p
& = \frac{\textup{d}}{\textup{d}t}\Big|_{t=0} \int_{\mathbb{R}^d} \big| (T^A_t f)(x) \big|^p \,\textup{d}x \\
& = -\frac{p^2}{2} \int_{\mathbb{R}^d} s^{p-2} \mathop{\textup{Re}}\Big\langle A(x)\xi, \ \xi + \Big(1-\frac{2}{p}\Big) \,\overline{\xi} \Big\rangle_{\mathbb{C}^d} \,\textup{d}x,
\end{align*}
where we used the shorthand notation
\begin{equation}\label{eq:sxitmp}
s=|f(x)|, \quad \xi=\overline{\mathop{\textup{sgn}}f(x)}\,\nabla f(x).
\end{equation}
From this it is easy to see that $\Delta_p(A)\geq0$ is an elegant sufficient condition for contractivity of the semigroup $(T_t^A)_{t>0}$ on $\textup{L}^{p}(\mathbb{R}^d)$; rigorous arguments and numerous details can be found in \cite[Section~7]{CarbonaroDragicevic20}.
Prior to \cite{CarbonaroDragicevic20}, Cialdea and Maz'ya \cite{CialdeaMazya05} studied \emph{$\textup{L}^p$-dissipativity} (see \cite[Definition~1]{CialdeaMazya05}) of the sesquilinear form \eqref{eq:sesqform}, which is, by a result of Nittka \cite{Nittka12}, equivalent to the aforementioned contractivity.
In the particular case when $\mathop{\textup{Im}}A$ is symmetric, their result \cite[Theorem~5]{CialdeaMazya05} claims that these are further equivalent to the condition
\begin{equation}\label{eq:CMLp}
|p-2| \big|\langle \mathop{\textup{Im}}A(x)\xi,\xi\rangle_{\mathbb{R}^d}\big| \leqslant 2(p-1)^{1/2} \langle \mathop{\textup{Re}}A(x)\xi,\xi\rangle_{\mathbb{R}^d}
\end{equation}
for all $x,\xi\in\mathbb{R}^d$.
Much more recently, Cialdea and Maz'ya \cite{CialdeaMazya21} introduced the concept of \emph{functional dissipativity} with respect to a general function $\Psi$ as before. Under the same assumption that $\mathop{\textup{Im}}A$ is symmetric, their result \cite[Theorem~1]{CialdeaMazya21} characterizes this property via the condition
\begin{equation}\label{eq:CMfunctional}
\Big|\Psi''(s)-\frac{\Psi'(s)}{s}\Big| \big|\langle \mathop{\textup{Im}}A(x)\xi,\xi\rangle_{\mathbb{R}^d}\big| \leqslant 2\Big(\frac{\Psi'(s)\Psi''(s)}{s}\Big)^{1/2} \langle \mathop{\textup{Re}}A(x)\xi,\xi\rangle_{\mathbb{R}^d}
\end{equation}
for all $x,\xi\in\mathbb{R}^d$ and $s\in(0,\infty)$.
To avoid any possible confusion, let us mention that the paper \cite{CialdeaMazya21} prefers to formulate \eqref{eq:CMfunctional} in terms of $\varphi(s)=\Psi'(s)/s$.
Exact relationship of this functional dissipativity to the contractivity of $(T_t^A)_{t>0}$ in Orlicz spaces has not been clarified yet, as the Luxemburg norms \eqref{eq:Luxnorms} have less straightforward definitions than the $\textup{L}^p$ norms.
However, one can still differentiate the nonlinear functional:
\begin{align*}
& \frac{\textup{d}}{\textup{d}t}\Big|_{t=0} \int_{\mathbb{R}^d} \Psi\big(\big| (T^A_t f)(x) \big|\big) \,\textup{d}x \\
& = -\frac{1}{2} \int_{\mathbb{R}^d} \Big(\Psi''(s)+\frac{\Psi'(s)}{s}\Big)\mathop{\textup{Re}}\Big\langle A(x)\xi,\, \xi + \frac{s\Psi''(s)-\Psi'(s)}{s\Psi''(s)+\Psi'(s)}\,\overline{\xi} \Big\rangle_{\mathbb{C}^d} \,\textup{d}x,
\end{align*}
where $s$ and $\xi$ are as in \eqref{eq:sxitmp}.
From this we see that $\Delta_{\Psi}(A)\geq0$ is in fact also a very natural condition, which might justify our choice for $\Delta_{\Psi}(A)$.
Much more on functional dissipativity can be found in \cite{CialdeaMazya21}, but these discussions are not strictly relevant here.
We are aiming at bilinear embeddings, which are significantly more involved and \eqref{eq:condPhi1} simply appears in the computations in Sections~\ref{sec:Hessians} and \ref{sec:Bellmanfn}.
\end{remark}

\begin{remark}\label{rem:simplythep}
Now, if $\Phi$ is a Young function satisfying \eqref{eq:PhiPsicond1}--\eqref{eq:PhiPsicond3} then the computation from \eqref{eq:weicrucomp} easily shows
\[ \Delta_{\Phi_\ast}(A) = \Delta_{\Phi}(A). \]
This is a generalization of the fact that $\Delta_p(A)$ is invariant under conjugation of the Lebesgue exponent $p$.
However, definition \eqref{eq:condPhi1} does not lead to a particularly novel concept for such $\Phi$, because it can be reduced to the mere $p$-ellipticity for an appropriate number $p$.
More precisely, $\Delta_{\Phi}(A) = \Delta_{p}(A)$ for the unique $p\in[2,\infty]$ such that
\[ \sup_{s\in(0,\infty)} \Big|\frac{s\Phi''(s)-\Phi'(s)}{s\Phi''(s)+\Phi'(s)}\Big| = 1-\frac{2}{p}. \]
If $\Phi$ additionally satisfies all standing assumptions from Subsection~\ref{subsec:Orlicz}, then the number $p$ simplifies as
\begin{equation}\label{eq:defofexpp}
p = \sup_{s\in(0,\infty)}\frac{s\Phi''(s)}{\Phi'(s)} + 1.
\end{equation}
We further recognize it as $\tilde{M}+1$, with the number $\tilde{M}$ given in \eqref{eq:quantities2}.
Consequently, we will avoid complications and formulate a condition in our main result simply in terms of $p$-ellipticity for the exponent \eqref{eq:defofexpp}.
\end{remark}

\subsection{The main result}
\label{subsec:formulationresults}
Finally, we can state the desired estimate. Recall quantities \eqref{eq:quantities1}--\eqref{eq:quantities4} from Subsection~\ref{subsec:Orlicz} and definitions \eqref{eq:condbigl}, \eqref{eq:condlittlel}, \eqref{eq:condpellip} from Subsection~\ref{subsec:divform}.

\begin{theorem}\label{thm:mainthm}
Suppose that $\Phi$ and ${\Phi_{\ast}}$ satisfy conditions \eqref{eq:PhiPsicond1}--\eqref{eq:Phicond3} and let $A,B\colon\mathbb{R}^d\to\mathbb{C}^{d\times d}$ be $p$-elliptic matrix functions with $\textup{L}^\infty$ coefficients, where $p=\tilde{M}+1$, i.e., $p$ is given by \eqref{eq:defofexpp}.
Denote
\begin{equation}\label{eq:cabconst}
C_{p}(A,B) := \frac{\max\{\Lambda(A),\Lambda(B)\}}{\min\{\Delta_{p}(A),\Delta_{p}(B)\} \min\{\lambda(A),\lambda(B)\}}
\end{equation}
and
\begin{equation}\label{eq:dppconst}
D(\Phi) := \max\Big\{1,\frac{M}{\tilde{m}}\Big\} \Big(\frac{\tilde{m}}{\tilde{M}} \frac{\tilde{M}-1}{\tilde{m}-1}\Big)^{1/2} .
\end{equation}
Then an Orlicz-space bilinear embedding,
\begin{equation}\label{eq:bilinorl}
\int_{0}^{\infty} \int_{\mathbb{R}^d} |(\nabla T^A_t f)(x)| \,|(\nabla T^B_t g)(x)| \,\textup{d}x \,\textup{d}t
\leqslant 40 \,C_{p}(A,B) \,D(\Phi) \,\|f\|_{\Phi} \|g\|_{{\Phi_{\ast}}},
\end{equation}
holds for any complex functions $f,g\in\textup{C}_{c}^{\infty}(\mathbb{R}^d)$.
\end{theorem}

A few comments on Theorem~\ref{thm:mainthm} could help to better orient the reader.

\begin{remark}[Constants]
Quantity \eqref{eq:dppconst} depends only on $\Phi$ (or, equivalently, on ${\Phi_{\ast}}$), while \eqref{eq:cabconst} depends on the ellipticity constants of $A$ and $B$ and on the exponent $p$, which, in turn, depends only on $\Phi$ again. The constant in \eqref{eq:bilinorl} depends on the ambient dimension $d$ in no other way than through these two quantities, so we can say that this estimate is \emph{dimension-free}. This is a desired property of all bilinear embeddings.
\end{remark}

\begin{remark}[Real case]
If $A$ and $B$ have real coefficients, then the $p$-ellipticity condition is satisfied automatically; recall \eqref{eq:pellinclusions}. We are not in a position to list the vast literature on estimates for real elliptic divergence-form operators, including many singular integral estimates as their special cases; see the references in \cite{CarbonaroDragicevic20,CarbonaroDragicevicKovacSkreb21}. The emphasis of the present paper is on the complex case.
\end{remark}

\begin{remark}[Duality]
\label{rem:duality}
By Remark~\ref{rem:dualnorms} and since $\Phi$ is doubling, the product $\|f\|_{\Phi} \|g\|_{{\Phi_{\ast}}}$ on the right hand side of \eqref{eq:bilinorl} can be rewritten as $\|f\|_{\Phi} \|g\|_{\Phi}^\ast$. That way \eqref{eq:bilinorl} can sometimes be viewed as an estimate on a single Orlicz space $\textup{L}^{\Phi}(\mathbb{R}^d)$ (see the definition in \cite{RaoRen91,HarjulehtoHasto19}) of a certain linear operator $\mathcal{L}$ such that the left hand side of \eqref{eq:bilinorl} dominates $|\langle\mathcal{L}f,g\rangle_{\textup{L}^2(\mathbb{R}^d)}|$.
\end{remark}

\begin{remark}[Applicability]
Estimate \eqref{eq:bilinorl} is a generalization of \cite[Theorem~1.1]{CarbonaroDragicevic20} by Carbonaro and Dragi\v{c}evi\'{c}, which was concerned with $\textup{L}^p$ and $\textup{L}^q$ norms only, i.e., with $\Phi$ and ${\Phi_{\ast}}$ given by \eqref{eq:LpLq} in Example~\ref{ex:Lpnorms}.
Indeed, the constant \eqref{eq:cabconst} is the same one appearing in their theorem, just formulated in a slightly different manner. Also, in the particular case \eqref{eq:LpLq} we easily compute \eqref{eq:dppconst} as
\[ D(\Phi) = \frac{p}{p-1} = q \leqslant 2, \]
so our constant in \eqref{eq:bilinorl} becomes the same one as in \cite{CarbonaroDragicevic20}, up to a factor $4$.
Theorem~\ref{thm:mainthm} also applies to the other examples given in Subsection~\ref{subsec:Orlicz}. In the case of Example~\ref{ex:LrlogL} the exact exponent $p$ is some number from $[r,r+1)$, so one can safely replace it by $r+1$. In Example~\ref{ex:powers1}, just as in Example~\ref{ex:Lpnorms}, this exponent is exactly the eponymous parameter $p$, while in Example~\ref{ex:powers2} it is equal to $q/(q-1)$, the conjugate exponent of $q$.
\end{remark}

\begin{remark}[Dehomogenization]
\label{rem:dehomogenization}
It is easy to see that estimate \eqref{eq:bilinorl} follows from (what could be called) a Young-function bilinear embedding,
\begin{align}
& \int_{0}^{\infty} \int_{\mathbb{R}^d} |(\nabla T^A_t f)(x)| \,|(\nabla T^B_t g)(x)| \,\textup{d}x \,\textup{d}t \nonumber \\
& \leqslant 20 \,C_{p}(A,B) \,D(\Phi) \Big( \int_{\mathbb{R}^d} \Phi(|f(x)|) \,\textup{d}x + \int_{\mathbb{R}^d} {\Phi_{\ast}}(|g(x)|) \,\textup{d}x \Big). \label{eq:bilinorl2}
\end{align}
Indeed, take arbitrary functions $f,g$ and arbitrary $\alpha,\beta\in(0,\infty)$ such that
\[ \int_{\mathbb{R}^d} \Phi\Big(\frac{|f(x)|}{\alpha}\Big) \,\textup{d}x \leqslant 1, \quad \int_{\mathbb{R}^d} {\Phi_{\ast}}\Big(\frac{|g(x)|}{\beta}\Big) \,\textup{d}x \leqslant 1. \]
By applying \eqref{eq:bilinorl2} to $f/\alpha$ and $g/\beta$ and using homogeneity of the left hand side, we conclude
\[ \int_{0}^{\infty} \int_{\mathbb{R}^d} |(\nabla T^A_t f)(x)| \,|(\nabla T^B_t g)(x)| \,\textup{d}x \,\textup{d}t
\leqslant 40 \,C_{p}(A,B) \,D(\Phi) \,\alpha \beta. \]
Taking infima over all such $\alpha$ and $\beta$ we derive \eqref{eq:bilinorl}.
Thus, we only need to establish \eqref{eq:bilinorl2}.
Advantages of such ``dehomogenization'' for proofs using the Bellman function technique were elaborated by Nazarov and Treil \cite[Section~8.1]{NazarovTreil96}, who called it the ``H\"{o}lder vs Young'' trick.  It is not at all lossy (up to unimportant constants) in the case of the $\textup{L}^p$ spaces, while here it simply seems to be the most natural thing to use.
\end{remark}

\begin{remark}[Interpolation]
In relation with Remark~\ref{rem:duality}, many particular cases and weaker forms of estimate \eqref{eq:bilinorl} are immediate consequences of \cite[Theorem~1.1]{CarbonaroDragicevic20}. Let us again take a linear operator $\mathcal{L}$ such that $|\langle\mathcal{L}f,g\rangle_{\textup{L}^2(\mathbb{R}^d)}|$ is less than or equal to the left hand side of \eqref{eq:bilinorl}.
The $p$-ellipticity assumption guarantees $\textup{L}^p\to\textup{L}^p$ and $\textup{L}^{q}\to\textup{L}^{q}$ estimates, where $q$ is the exponent conjugate to $p$. A collection of Orlicz spaces $\textup{L}^\Phi$ is ``squeezed between'' $\textup{L}^p$ and $\textup{L}^{q}$, so that certain interpolation arguments can cheaply provide the estimate $\textup{L}^{\Phi}\to\textup{L}^{\Phi}$. However, these arguments cannot recover Theorem~\ref{thm:mainthm} in its full generality.

Indeed, \emph{real} (i.e., Marcinkiewicz-type) Orlicz-space interpolation \cite{Zygmund56,Torchinsky76,Cianchi98} applies as soon as the Young function $\Phi$ is ``sufficiently far'' from the powers $s\mapsto s^p$ and $s\mapsto s^{q}$.
Cianchi \cite{Cianchi98} provided a definite result on the topic and gave a precise description of all Young functions $\Phi$ such that (what is nowadays usually called) restricted weak type $(p,p)$ and $(q,q)$ bounds generally imply the strong bound $\textup{L}^{\Phi}\to\textup{L}^{\Phi}$. Just a single necessary condition (out of many) from his paper in our case reads
\begin{equation}\label{eq:Cianchicond}
\int_1^{\infty} \frac{\Phi(s)}{s^{p+1}} \,\textup{d}s < \infty.
\end{equation}
Note that \eqref{eq:Cianchicond} is not satisfied for \eqref{eq:spsr}, even if we only take $\varepsilon=1$, so Marcinkiewicz-type interpolation cannot give our estimate \eqref{eq:bilinorl} in the case of Example~\ref{ex:powers1}.
Moreover, any real interpolation argument one could think of would give a constant that necessarily blows up as $\varepsilon\to0^+$. On the other hand, we see that \eqref{eq:cabconst} and \eqref{eq:dppconst} are independent of $\varepsilon$, and so is our main estimate.

Strictly speaking, \emph{complex} Orlicz-space interpolation \cite{Kreinetal82,KarlovichMaligranda01} is not applicable simply because the left hand side of \eqref{eq:bilinorl} is only bi-sublinear and not bilinear in $f$ and $g$.
However, in many applications of bilinear embeddings we only need to control the bilinear form $(f,g)\mapsto\langle \mathcal{L}f, g \rangle_{\textup{L}^2(\mathbb{R}^d)}$ for some linear operator $\mathcal{L}$, as mentioned earlier. Then Orlicz-space interpolation of linear operators can be useful, albeit still with certain limitations; see \cite[Theorem~5.1]{KarlovichMaligranda01}. Nevertheless, even in particular situations when complex interpolation does apply, it is still interesting to have a direct proof of the estimate $\textup{L}^{\Phi}\to\textup{L}^{\Phi}$.
\end{remark}

\begin{remark}[Sharpness]
One might initially feel dissatisfied by the fact that Theorem~\ref{thm:mainthm} does not apply to any Orlicz spaces that are ``close'' to $\textup{L}^1(\mathbb{R}^d)$ or $\textup{L}^{\infty}(\mathbb{R}^d)$, as such endpoint estimates are often interesting in harmonic analysis.
However, here we have every right to question the mere possibility of endpoint estimates, because already a more basic result on semigroups \cite{CialdeaMazya21} fails without the very restrictive assumption \eqref{eq:CMfunctional}, which transforms into \eqref{eq:CMLp} and \eqref{eq:defofexpp} for a finite $p$.
Indeed, our main estimate \eqref{eq:bilinorl} does not allow such endpoint generalizations either. Let us give a sketchy argument to support this claim.

Suppose that $\textup{L}^\Phi$ lies ``at the end'' of the $\textup{L}^p$ range for $p\in[2,\infty)$ in the sense that all these $\textup{L}^p$ spaces are interpolation spaces for linear operators between $\textup{L}^2$ and $\textup{L}^\Phi$.
Also suppose that an estimate of type \eqref{eq:bilinorl} holds for this Young function $\Phi$ and for the very special matrix-functions
\[ A=e^{i\phi}I_d,\quad B=e^{-i\phi}I_d,\quad \phi\in(-\pi/2,\pi/2). \]
Combining Remark~\ref{rem:duality} with considerations from \cite[Section~1.6]{CarbonaroDragicevic20}, we see that this estimate would imply
\[ \sup_{t\in(0,\infty)} \big\| \exp(t e^{i\phi} \Delta_d) \big\|_{\textup{L}^\Phi\to\textup{L}^\Phi} \leqslant C_{\Phi,\phi}, \]
where $\Delta_d$ is the $d$-dimensional Laplacian and $C_{\Phi,\phi}$ is a finite constant.
Interpolation gives
\begin{equation}\label{eq:Lpsharpness}
\sup_{t\in(0,\infty)} \big\| \exp(t e^{i\phi} \Delta_d) \big\|_{\textup{L}^p\to\textup{L}^p} \leqslant C_{p,\phi}
\end{equation}
for every $p\in[2,\infty)$ with a constant $C_{p,\phi}$ depending on $p$ and $\phi$, but independent of the ambient dimension $d$.
However, \cite[Theorem~6.2]{CarbonaroDragicevic20} evaluates the left hand side of \eqref{eq:Lpsharpness} and shows that it blows up as $d\to\infty$ whenever $|\phi|>\arccos|1-2/p|$.
Whichever conditions we impose on our complex matrix functions in order to have a desired Orlicz-space estimate, we expect them to be satisfied at least for some $\phi\neq0$, but then we arrive at a contradiction by choosing a sufficiently large $p$.
A similar argument applies to Orlicz spaces that lie at the left end of the $\textup{L}^p$ range for $p\in(1,2]$.
\end{remark}

\begin{remark}[Bilinear vs.\@ multilinear]
Let us conclude with a remark that this paper is very bi-(sub)li\-ne\-ar in nature. It benefited from concentrating on estimates that simultaneously involve two complementary Young functions, $\Phi$ and ${\Phi_{\ast}}$.
A recent paper by Carbonaro, Dragi\v{c}evi\'{c}, and the present authors \cite{CarbonaroDragicevicKovacSkreb21} studied trilinear embeddings in $\textup{L}^p$ spaces.
It is unclear to us how to even formulate any Orlicz-space multi-(sub)li\-ne\-ar embeddings, as we do not know how to meaningfully define convex conjugation for more than two Young functions.
\end{remark}


\section{Generalized Hessians}
\label{sec:Hessians}
A quantity introduced by Carbonaro and Dragi\v{c}evi\'{c} in \cite[Subsection~2.2]{CarbonaroDragicevic20} will play a crucial role later in the proof.
For $\mathfrak{X}\colon\mathbb{C}^2\to[0,\infty)$, $A,B\in\mathbb{C}^{d\times d}$, $(u,v)\in\mathbb{C}^2$, and $(\zeta,\eta)\in(\mathbb{C}^d)^2$
we define the \emph{generalized Hessian} of $\mathfrak{X}$ with respect to $A,B$,
\[ H^{A,B}_{\mathfrak{X}}[(u,v);(\zeta,\eta)], \]
as the standard inner product of
\[ \big(\textup{Hess}(\mathfrak{X};(u,v)) \otimes I_{d}\big)
\begin{bmatrix}\mathop{\textup{Re}}\zeta \\ \mathop{\textup{Im}}\zeta \\ \mathop{\textup{Re}}\eta \\ \mathop{\textup{Im}}\eta \end{bmatrix} \in(\mathbb{R}^d)^4 \]
and
\[ \begin{bmatrix}\mathop{\textup{Re}}A & -\mathop{\textup{Im}}A & \mathbf{0} & \mathbf{0} \\ \mathop{\textup{Im}}A & \mathop{\textup{Re}}A & \mathbf{0} & \mathbf{0} \\
\mathbf{0} & \mathbf{0} & \mathop{\textup{Re}}B & -\mathop{\textup{Im}}B \\ \mathbf{0} & \mathbf{0} & \mathop{\textup{Im}}B & \mathop{\textup{Re}}B \end{bmatrix}
\begin{bmatrix}\mathop{\textup{Re}}\zeta \\ \mathop{\textup{Im}}\zeta \\ \mathop{\textup{Re}}\eta \\ \mathop{\textup{Im}}\eta \end{bmatrix} \in(\mathbb{R}^d)^4. \]
Here one has to interpret $\textup{Hess}(\mathfrak{X};(u,v))$ as the $4\times 4$ real Hessian matrix of the function
\[ \mathbb{R}^4\to\mathbb{R}, \quad (u_r,u_i,v_r,v_i) \mapsto \mathfrak{X}(u_r+iu_i,v_r+iv_i). \]
Operation $\otimes$ is the \emph{Kronecker} (a.k.a.\@ \emph{tensor}) \emph{product} of matrices.
We also introduce
\[ \widetilde{H}^{A,B}_{\mathfrak{X}}[(u,v);(\zeta,\eta)] := H^{A,B}_{\mathfrak{X}}\Big[(u,v);\Big(\frac{u}{|u|}\zeta,\frac{v}{|v|}\eta\Big)\Big], \]
as the replacements
\begin{equation}\label{eq:replacements}
\zeta\rightarrow\frac{u}{|u|}\zeta,\quad \eta\rightarrow\frac{v}{|v|}\eta,
\end{equation}
will later significantly simplify numerous expressions.

The following lemma is much in the spirit of computations from \cite{CarbonaroDragicevic20} and \cite{CarbonaroDragicevicKovacSkreb21}. However, in those papers properties of power functions $s\mapsto |s|^p$ are much appreciated, while here we will be dealing with more general nonlinear functions.

\begin{lemma}\label{lm:hessians}
If we define
\begin{align*}
\mathfrak{X}_1(u,v) & := P(|u|) + Q(|v|), \\
\mathfrak{X}_2(u,v) & := |u|^2 R(|v|)
\end{align*}
for some $\textup{C}^2$ functions $P,Q,R\colon(0,\infty)\to\mathbb{R}$, then the following formulas hold for any $A,B\in\mathbb{C}^{d\times d}$, $(u,v)\in(\mathbb{C}\setminus\{0\})^2$, and $(\zeta,\eta)\in(\mathbb{C}^d)^2$:
{\allowdisplaybreaks
\begin{align*}
& \widetilde{H}^{A,B}_{\mathfrak{X}_1}[(u,v);(\zeta,\eta)] \\
& = \mathop{\textup{Re}}\bigg\langle A\zeta, \ \frac{1}{2}\Big(P''(|u|)+\frac{P'(|u|)}{|u|}\Big) \zeta + \frac{1}{2}\Big(P''(|u|)-\frac{P'(|u|)}{|u|}\Big) \bar{\zeta} \bigg\rangle_{\mathbb{C}^d} \\
& \quad + \mathop{\textup{Re}}\bigg\langle B\eta, \ \frac{1}{2}\Big(Q''(|v|)+\frac{Q'(|v|)}{|v|}\Big) \eta + \frac{1}{2}\Big(Q''(|v|)-\frac{Q'(|v|)}{|v|}\Big) \bar{\eta} \bigg\rangle_{\mathbb{C}^d}, \\
& \widetilde{H}^{A,B}_{\mathfrak{X}_2}[(u,v);(\zeta,\eta)] \\
& = \mathop{\textup{Re}}\Big\langle A\zeta, \ 2R(|v|) \zeta + |u| R'(|v|) \big(\eta + \bar{\eta}\big)  \Big\rangle_{\mathbb{C}^d} \\
& \quad + \mathop{\textup{Re}}\bigg\langle B\eta, \ |u| R'(|v|) \big(\zeta+\bar{\zeta}\big)
+ \frac{|u|^2}{2}\Big(R''(|v|)+\frac{R'(|v|)}{|v|}\Big) \eta + \frac{|u|^2}{2}\Big(R''(|v|)-\frac{R'(|v|)}{|v|}\Big) \bar{\eta} \bigg\rangle_{\mathbb{C}^d}.
\end{align*}}
\end{lemma}

\begin{proof}
In the case of $\mathfrak{X}_1$ the Hessian matrix $\textup{Hess}(\mathfrak{X}_1;(u,v))$ in the variables $u_r,u_i,v_r,v_i$ is easily evaluated to be the direct sum of matrices
\[ \begin{bmatrix}
P''(|u|)\frac{u_r^2}{|u|^2} + \frac{P'(|u|)}{|u|}\frac{u_i^2}{|u|^2} & (P''(|u|)-\frac{P'(|u|)}{|u|})\frac{u_r u_i}{|u|^2} \\
(P''(|u|)-\frac{P'(|u|)}{|u|})\frac{u_r u_i}{|u|^2} & \frac{P'(|u|)}{|u|}\frac{u_r^2}{|u|^2} + P''(|u|)\frac{u_i^2}{|u|^2}
\end{bmatrix} \]
and
\[ \begin{bmatrix}
Q''(|v|)\frac{v_r^2}{|v|^2} + \frac{Q'(|v|)}{|v|}\frac{v_i^2}{|v|^2} & (Q''(|v|)-\frac{Q'(|v|)}{|v|})\frac{v_r v_i}{|v|^2} \\
(Q''(|v|)-\frac{Q'(|v|)}{|v|})\frac{v_r v_i}{|v|^2} & \frac{Q'(|v|)}{|v|}\frac{v_r^2}{|v|^2} + Q''(|v|)\frac{v_i^2}{|v|^2}
\end{bmatrix}. \]
It is then tensored with the identity matrix $I_d$, multiplied with the column-vector
\begin{equation}\label{eq:auxcolvec1}
\begin{bmatrix} \mathop{\textup{Re}}\zeta & \mathop{\textup{Im}}\zeta & \mathop{\textup{Re}}\eta & \mathop{\textup{Im}}\eta \end{bmatrix}^{\textup{T}},
\end{equation}
and the result is interpreted as a vector in $(\mathbb{C}^d)^2$, rather than in $(\mathbb{R}^d)^4$:
\begin{align*}
& \begin{bmatrix}
\frac{u}{|u|} P''(|u|) \mathop{\textup{Re}}(\frac{\bar{u}}{|u|}\zeta) + i\frac{u}{|u|} \frac{P'(|u|)}{|u|} \mathop{\textup{Im}}(\frac{\bar{u}}{|u|}\zeta) \\
\frac{v}{|v|} Q''(|v|) \mathop{\textup{Re}}(\frac{\bar{v}}{|v|}\eta) + i\frac{v}{|v|} \frac{Q'(|v|)}{|v|} \mathop{\textup{Im}}(\frac{\bar{v}}{|v|}\eta)
\end{bmatrix} \\
& = \begin{bmatrix} \frac{u}{|u|} I_d & \mathbf{0} \\ \mathbf{0} & \frac{v}{|v|} I_d \end{bmatrix}
\begin{bmatrix}
\frac{1}{2}(P''(|u|) + \frac{P'(|u|)}{|u|}) (\frac{\bar{u}}{|u|}\zeta) +
\frac{1}{2}(P''(|u|) - \frac{P'(|u|)}{|u|}) \overline{(\frac{\bar{u}}{|u|}\zeta)} \\
\frac{1}{2}(Q''(|v|) + \frac{Q'(|v|)}{|v|}) (\frac{\bar{v}}{|v|}\eta) +
\frac{1}{2}(Q''(|v|) - \frac{Q'(|v|)}{|v|}) \overline{(\frac{\bar{v}}{|v|}\eta)}
\end{bmatrix} .
\end{align*}
Taking the inner product
\[ \langle\cdot,\cdot\rangle_{(\mathbb{R}^d)^4} = \mathop{\textup{Re}} \langle\cdot,\cdot\rangle_{(\mathbb{C}^d)^2} \]
with the vector
\begin{equation}\label{eq:auxcolvec2}
\begin{bmatrix} A & \mathbf{0} \\ \mathbf{0} & B \end{bmatrix} \begin{bmatrix} \zeta \\ \eta \end{bmatrix}
= \begin{bmatrix} A \zeta \\ B \eta \end{bmatrix}
= \begin{bmatrix} \frac{u}{|u|} I_d & \mathbf{0} \\ \mathbf{0} & \frac{v}{|v|} I_d \end{bmatrix}
\begin{bmatrix} A (\frac{\bar{u}}{|u|}\zeta) \\ B (\frac{\bar{v}}{|v|}\eta) \end{bmatrix}
\end{equation}
we obtain a formula for $H^{A,B}_{\mathfrak{X}_1}[(u,v);(\zeta,\eta)]$.
It remains to change the variables as in \eqref{eq:replacements}.

In the case of $\mathfrak{X}_2$ the Hessian matrix $\textup{Hess}(\mathfrak{X}_2;(u,v))$ is
\[ \begin{bmatrix}
2R(|v|)
& 0
& 2 u_r R'(|v|)\frac{v_r}{|v|}
& 2 u_r R'(|v|)\frac{v_i}{|v|} \\
0
& 2R(|v|)
& 2 u_i R'(|v|)\frac{v_r}{|v|}
& 2 u_i R'(|v|)\frac{v_i}{|v|} \\
2 u_r R'(|v|)\frac{v_r}{|v|}
& 2 u_i R'(|v|)\frac{v_r}{|v|}
& |u|^2 (R''(|v|)\frac{v_r^2}{|v|^2} + \frac{R'(|v|)}{|v|}\frac{v_i^2}{|v|^2})
& |u|^2 (R''(|v|)-\frac{R'(|v|)}{|v|})\frac{v_r v_i}{|v|^2} \\
2 u_r R'(|v|)\frac{v_i}{|v|}
& 2 u_i R'(|v|)\frac{v_i}{|v|}
& |u|^2 (R''(|v|)-\frac{R'(|v|)}{|v|})\frac{v_r v_i}{|v|^2}
& |u|^2 (\frac{R'(|v|)}{|v|}\frac{v_r^2}{|v|^2} + R''(|v|)\frac{v_i^2}{|v|^2})
\end{bmatrix}. \]
This matrix tensored with $I_d$ and multiplied with the column-vector \eqref{eq:auxcolvec1} gives:
\[ \begin{bmatrix} \frac{u}{|u|} I_d & \mathbf{0} \\ \mathbf{0} & \frac{v}{|v|} I_d \end{bmatrix}
\begin{bmatrix}
2R(|v|) \frac{\bar{u}}{|u|}\zeta + 2 |u| R'(|v|) \mathop{\textup{Re}}(\frac{\bar{v}}{|v|}\eta) \\
2 |u| R'(|v|) \mathop{\textup{Re}}(\frac{\bar{u}}{|u|}\zeta) + |u|^2 R''(|v|) \mathop{\textup{Re}}(\frac{\bar{v}}{|v|}\eta) + i |u|^2 \frac{R'(|v|)}{|v|}  \mathop{\textup{Im}}(\frac{\bar{v}}{|v|}\eta)
\end{bmatrix}. \]
The desired formula follows by taking the inner product with \eqref{eq:auxcolvec2} and substituting \eqref{eq:replacements}.
\end{proof}


\section{The Bellman function}
\label{sec:Bellmanfn}
Suppose that $\Phi$ and ${\Phi_{\ast}}$ are as in the formulation of Theorem~\ref{thm:mainthm}, i.e., they fulfil conditions \eqref{eq:PhiPsicond1}--\eqref{eq:Phicond3} and, thus, also conditions/properties \eqref{eq:bijections}--\eqref{eq:intPsiupper}.
In particular, by recalling \eqref{eq:bijections} we observe that the surface
\[ \mathcal{Y} := \{(u,v)\in\mathbb{C}^2 : |v|=\Phi'(|u|)\} = \{(u,v)\in\mathbb{C}^2 : |u|=\Phi_{\ast}'(|v|)\} \]
splits its complement into two regions: the ``lower'' region
\[ \mathcal{Y}_{\downarrow} := \{(u,v)\in\mathbb{C}^2 : |v|<\Phi'(|u|)\} = \{(u,v)\in\mathbb{C}^2 : |u|>\Phi_{\ast}'(|v|)\} \]
and the ``upper'' region
\[ \mathcal{Y}_{\uparrow} := \{(u,v)\in\mathbb{C}^2 : |v|>\Phi'(|u|)\} = \{(u,v)\in\mathbb{C}^2 : |u|<\Phi_{\ast}'(|v|)\}. \]

Also suppose that $A,B\colon\mathbb{R}^d\to\mathbb{C}^{d\times d}$ are matrix functions as in the statement of Theorem~\ref{thm:mainthm}.
By Remark~\ref{rem:simplythep} from Subsection~\ref{subsec:divform} we have
\[ \Delta_{\Phi}(A) = \Delta_p(A),\quad \Delta_{{\Phi_{\ast}}}(B) = \Delta_p(B) \]
for $p=\tilde{M}+1$.

The so-called \emph{Bellman function method} relies on boundedness and convexity properties of a carefully chosen auxiliary function; see the seminal paper by Nazarov, Treil, and Volberg \cite{NazarovTreilVolberg99} and classical expository papers \cite{NazarovTreil96} and \cite{NazarovTreilVolberg01}.
We need to construct a Bellman function relevant to the present problem. Let us define $\mathfrak{X}\colon\mathbb{C}^2\to[0,\infty)$ as
\begin{equation}\label{eq:mainBellman}
\mathfrak{X}(u,v) := \begin{cases}
{\displaystyle (1+\delta)\big(\Phi(|u|) + {\Phi_{\ast}}(|v|)\big) +  \delta |u|^2 \int_{0}^{|u|} \frac{\Phi'(s)\,\textup{d}s}{s^2}} & \text{ for } (u,v)\in\mathcal{Y}_{\downarrow}\cup\mathcal{Y}, \\
{\displaystyle \Phi(|u|) + {\Phi_{\ast}}(|v|) + \delta |u|^2\int_{0}^{|v|}\frac{\textup{d}s}{\Phi_{\ast}'(s)}} & \text{ for } (u,v)\in\mathcal{Y}_{\uparrow},
\end{cases}
\end{equation}
where
\begin{equation}\label{eq:choiceofdelta}
\delta := \frac{\tilde{m}-1}{\tilde{m}} \min\Big\{ \frac{\Delta_{p}(A)}{8\Lambda(A)}, \frac{\Delta_{p}(B)}{4\Lambda(B)}, \frac{\lambda(A)\Delta_{p}(B)}{100\max\{ \Lambda(A)^2,\Lambda(B)^2 \}} \Big\}.
\end{equation}
Note that $s\mapsto\Phi'(s)/s^2$ and $s\mapsto1/\Phi_{\ast}'(s)$ are integrable in a neighborhood of $s=0$, thanks to \eqref{eq:intPhiupper} and \eqref{eq:intPsiupper}.

In the particular case of mutually conjugate Lebesgue space Young functions \eqref{eq:LpLq}, formula \eqref{eq:mainBellman} simplifies as
\begin{equation}\label{eq:particularBellman}
\frac{|u|^p}{p} + \frac{|v|^q}{q} + \frac{\delta}{2-q}\times \begin{cases}
\frac{2}{p}|u|^p + \big(\frac{2}{q}-1\big)|v|^q & \text{ for } |u|^p\geqslant |v|^q, \\
|u|^2 |v|^{2-q} & \text{ for } |u|^p<|v|^q,
\end{cases}
\end{equation}
which is just a minor modification of (a two-variable version of) the Bellman function from \cite[Section~8]{NazarovTreil96}.
At the first sight there seem to be many possibilities for $\mathfrak{X}$ generalizing \eqref{eq:particularBellman}, but inspection of desired properties below narrows down the choice severely.
Thus, even if the above choice for $\mathfrak{X}$ might not be the most obvious one, we find it necessary and somewhat canonical; see Remark~\ref{rem:Bellmannmotivation} for a reasoning that lead us naturally to the above formula.

The main task is to prove several crucial estimates for this function $\mathfrak{X}$.
Carbonaro and Dragi\v{c}evi\'{c} \cite{CarbonaroDragicevic20} established the required estimates in the particular case of power functions \eqref{eq:LpLq}, namely for a minor variant of \eqref{eq:particularBellman}.
This required a series of upgrades from its basic properties discussed already in \cite{NazarovTreil96} to their nontrivial extensions and additions addressed gradually in \cite{DragicevicVolberg12,CarbonaroDragicevic17,CarbonaroDragicevic20}.
The power functions played an important role in these papers, as emphasized already in the title of \cite{CarbonaroDragicevic20}.
Here we need to redo the calculations, as we are now working with more general $\Phi$ and ${\Phi_{\ast}}$.

We begin with some smoothness of $\mathfrak{X}$.

\begin{lemma}\label{lm:XisC1}
The function $\mathfrak{X}$ is $\textup{C}^1$ on the whole domain $\mathbb{C}^2\equiv\mathbb{R}^4$. Moreover, it is $\textup{C}^2$ on $(\mathbb{C}\setminus\{0\})^2 \setminus\mathcal{Y}$ and its second-order partial derivatives are locally integrable, i.e., they are integrable on every bounded measurable subset of $(\mathbb{C}\setminus\{0\})^2 \setminus\mathcal{Y}$.
\end{lemma}

\begin{proof}
Define $\mathfrak{B}\colon[0,\infty)^2\to[0,\infty)$ by
\[ \mathfrak{B}(x,y) := \begin{cases}
{\displaystyle (1+\delta)\big(\Phi(x) + {\Phi_{\ast}}(y)\big) +  \delta x^2 \int_{0}^{x} \frac{\Phi'(s)\,\textup{d}s}{s^2}} & \text{ for } y\leqslant\Phi'(x),\text{ i.e., } x\geqslant\Phi_{\ast}'(y),\\
{\displaystyle \Phi(x) + {\Phi_{\ast}}(y) + \delta x^2\int_{0}^{y}\frac{\textup{d}s}{\Phi_{\ast}'(s)}} & \text{ for } y>\Phi'(x),\text{ i.e., } x<\Phi_{\ast}'(y),
\end{cases} \]
where $\delta$ is as in \eqref{eq:choiceofdelta}. Thus, $\mathfrak{X}(u,v)=\mathfrak{B}(|u|,|v|)$. In order to see that $\mathfrak{X}$ is continuous on $\mathbb{C}^2$, it is sufficient to verify that $\mathfrak{B}$ is continuous on $[0,\infty)^2$. Each of the two defining formulas for $\mathfrak{B}$ is clearly continuous on $[0,\infty)^2$, so it remains to see that they coincide on the critical curve
\begin{equation}\label{eq:criticalcurve}
\{(x,y)\in[0,\infty)^2 : y=\Phi'(x)\} = \{(x,y)\in[0,\infty)^2 : x=\Phi_{\ast}'(y)\}.
\end{equation}
On this curve we have
\begin{align}
& x^2\int_{0}^{y}\frac{\textup{d}t}{\Phi_{\ast}'(t)} = x^2\int_{0}^{\Phi'(x)}\frac{\textup{d}t}{\Phi_{\ast}'(t)} = \Big[\begin{array}{c}t=\Phi'(s)\\ \textup{d}t=\Phi''(s)\,\textup{d}s\end{array}\Big] \nonumber \\
& = x^2\int_{0}^{x}\frac{\Phi''(s)\,\textup{d}s}{s} = \Big[\text{integration by parts and \eqref{eq:Phicond1}}\Big] \nonumber \\
& = x y + x^2\int_{0}^{x}\frac{\Phi'(s)\,\textup{d}s}{s^2}
= \Phi(x) + {\Phi_{\ast}}(y) + x^2\int_{0}^{x}\frac{\Phi'(s)\,\textup{d}s}{s^2}. \label{eq:integralcomput}
\end{align}
Here we used the well-known fact that Young's inequality \eqref{eq:Youngsineq} becomes an equality when the pair $(s,t)$ lies on the critical curve, which also easily follows from \eqref{eq:complementary}.
This confirms the continuity of $\mathfrak{B}$ and thus also of $\mathfrak{X}$.

Now we prove that all four first-order partial derivatives of $\mathfrak{X}$ exist and are continuous on $\mathbb{C}^2$.
Since
\begin{equation}\label{eq:XtoB1}
\partial_{u_r}\mathfrak{X}(u_r+iu_i,v_r+iv_i) = \partial_1\mathfrak{B}(|u|,|v|)\frac{u_r}{|u|}
\end{equation}
and similar equalities hold for other derivatives,
it is sufficient to show that $\mathfrak{B}$ is $\textup{C}^1$ on $(0,\infty)^2$, that its partial derivatives $\partial_1\mathfrak{B}(x,y)$, $\partial_2\mathfrak{B}(x,y)$ continuously extend to $[0,\infty)^2$, and that
\begin{align}
\lim_{(0,\infty)^2\ni(x,y)\to(0,y_0)} \partial_1\mathfrak{B}(x,y) & = 0, \label{eq:partlimits1} \\
\lim_{(0,\infty)^2\ni(x,y)\to(x_0,0)} \partial_2\mathfrak{B}(x,y) & = 0 \label{eq:partlimits2}
\end{align}
for any $x_0,y_0\in[0,\infty)$.
In fact, we will prove stronger statements,
\begin{align}
\partial_1\mathfrak{B}(x,y) & = O(x) \quad\text{as } x\to0^+, \text{ locally uniformly in } y\in[0,\infty), \label{eq:partlimits3} \\
\partial_2\mathfrak{B}(x,y) & = O(\Phi_{\ast}'(y)) \quad\text{as } y\to0^+, \text{ locally uniformly in } x\in[0,\infty), \label{eq:partlimits4}
\end{align}
which respectively imply \eqref{eq:partlimits1} and \eqref{eq:partlimits2}.
Partial derivatives of $\mathfrak{B}$ are
\begin{equation}\label{eq:parB1}
\partial_1\mathfrak{B}(x,y) = \begin{cases}
{\displaystyle (1+2\delta) \Phi'(x) +  2\delta x \int_{0}^{x} \frac{\Phi'(s)\,\textup{d}s}{s^2}} & \text{ for } 0<y<\Phi'(x),\\
{\displaystyle \Phi'(x) +  2\delta x \int_{0}^{y} \frac{\textup{d}s}{\Phi_{\ast}'(s)}} & \text{ for } 0<x<\Phi_{\ast}'(y)
\end{cases}
\end{equation}
and
\begin{equation}\label{eq:parB2}
\partial_2\mathfrak{B}(x,y) = \begin{cases}
{\displaystyle (1+\delta) \Phi_{\ast}'(y)} & \text{ for } 0<y<\Phi'(x),\\
{\displaystyle \Phi_{\ast}'(y) + \frac{\delta x^2}{\Phi_{\ast}'(y)}} & \text{ for } 0<x<\Phi_{\ast}'(y).
\end{cases}
\end{equation}
The very same computation \eqref{eq:integralcomput} shows that the two cases from \eqref{eq:parB1} coincide on the critical curve \eqref{eq:criticalcurve}.
Also, both cases from \eqref{eq:parB2} simplify on the curve as $(1+\delta)x$, confirming that $\mathfrak{B}$ is $\textup{C}^1$ but, so far, only in the open quadrant $(0,\infty)^2$.
It is clear that \eqref{eq:parB1} and \eqref{eq:parB2} continuously extend to the closed first quadrant $[0,\infty)^2$.
It is also straightforward to verify \eqref{eq:partlimits3} and \eqref{eq:partlimits4} and, while doing so, we need to remember the limit from condition \eqref{eq:Phicond1}.
Finally, we recall \eqref{eq:XtoB1} and similar formulas for the other derivatives of $\mathfrak{X}$.

Regarding second-order derivatives, we have
\begin{align*}
& \partial_{u_r}^2\mathfrak{X}(u_r+iu_i,v_r+iv_i) = \partial_1^2\mathfrak{B}(|u|,|v|)\frac{u_r^2}{|u|^2} + \partial_1\mathfrak{B}(|u|,|v|)\frac{u_i^2}{|u|^3}, \\
& \partial_{u_r}\partial_{u_i}\mathfrak{X}(u_r+iu_i,v_r+iv_i) = \partial_1^2\mathfrak{B}(|u|,|v|)\frac{u_r u_i}{|u|^2} - \partial_1\mathfrak{B}(|u|,|v|)\frac{u_r u_i}{|u|^3}, \\
& \partial_{u_r}\partial_{v_r}\mathfrak{X}(u_r+iu_i,v_r+iv_i) = \partial_1\partial_2\mathfrak{B}(|u|,|v|)\frac{u_r v_r}{|u| |v|},
\end{align*}
etc.
Second partial derivatives of $\mathfrak{B}$ are
\begin{align*}
& \partial_1^2\mathfrak{B}(x,y) = \begin{cases}
{\displaystyle (1+2\delta) \Phi''(x) + 2\delta \frac{\Phi'(x)}{x} + 2\delta\int_{0}^{x} \frac{\Phi'(s)\,\textup{d}s}{s^2}} & \text{ for } 0<y<\Phi'(x),\\
{\displaystyle \Phi''(x) +  2\delta \int_{0}^{y} \frac{\textup{d}s}{\Phi_{\ast}'(s)}} & \text{ for } 0<x<\Phi_{\ast}'(y),
\end{cases} \\
& \partial_1\partial_2\mathfrak{B}(x,y) = \begin{cases}
{\displaystyle 0} & \text{ for } 0<y<\Phi'(x),\\
{\displaystyle \frac{2\delta x}{\Phi_{\ast}'(y)}} & \text{ for } 0<x<\Phi_{\ast}'(y),
\end{cases} \\
& \partial_2^2\mathfrak{B}(x,y) = \begin{cases}
{\displaystyle (1+\delta) \Phi_{\ast}''(y)} & \text{ for } 0<y<\Phi'(x),\\
{\displaystyle \Phi_{\ast}''(y) - \frac{\delta x^2 \Phi_{\ast}''(y)}{\Phi_{\ast}'(y)^2}} & \text{ for } 0<x<\Phi_{\ast}'(y).
\end{cases}
\end{align*}
From these expressions and \eqref{eq:intPhiupper}, \eqref{eq:partlimits3}, \eqref{eq:partlimits4}, we easily conclude that $\partial_{u_r}^2\mathfrak{X}(u,v)$, $\partial_{u_r}\partial_{u_i}\mathfrak{X}(u,v)$, $\partial_{u_i}^2\mathfrak{X}(u,v)$ are
\[ O(\Phi''(|u|)) +O\Big(\frac{\Phi'(|u|)}{|u|}\Big) + O(1) \quad\text{as } u\to0, \]
locally uniformly in $v$, that $\partial_{u_r}\partial_{v_r}\mathfrak{X}(u,v)$, $\partial_{u_r}\partial_{v_i}\mathfrak{X}(u,v)$, $\partial_{u_i}\partial_{v_r}\mathfrak{X}(u,v)$, $\partial_{u_i}\partial_{v_i}\mathfrak{X}(u,v)$ are bounded, and that $\partial_{v_r}^2\mathfrak{X}(u,v)$, $\partial_{v_r}\partial_{v_i}\mathfrak{X}(u,v)$, $\partial_{v_i}^2\mathfrak{X}(u,v)$ are
\[ O(\Phi_{\ast}''(|v|)) +O\Big(\frac{\Phi_{\ast}'(|v|)}{|v|}\Big) \quad\text{as } v\to0, \]
locally uniformly in $u$.
Recalling \eqref{eq:Phicond3} and \eqref{eq:Psicond3}
we conclude that the second-order partial derivatives of $\mathfrak{X}$ are integrable in some neighborhood of each point of the two-dimensional coordinate planes $u=0$ and $v=0$.
Since the derivatives are obviously locally bounded outside of these two planes, the proof is complete.
\end{proof}

Let us proceed with an upper bound on $\mathfrak{X}$.

\begin{lemma}
\label{lm:Xupper}
The function $\mathfrak{X}$ satisfies
\begin{equation}\label{eq:Xupperbound}
\mathfrak{X}(u,v) \leqslant 2\max\Big\{1,\frac{M}{\tilde{m}}\Big\} \big(\Phi(|u|) + {\Phi_{\ast}}(|v|)\big)
\end{equation}
for $(u,v)\in\mathbb{C}^2$.
\end{lemma}

\begin{proof}
The estimate will be verified separately in the two regions.

\emph{Region $\mathcal{Y}_{\downarrow}\cup\mathcal{Y}$.}
In this region, \eqref{eq:intPhiupper} and \eqref{eq:quantities1} give
\[ \mathfrak{X}(u,v) \leqslant \Big(1+\delta+\delta\frac{M}{\tilde{m}-1}\Big)\Phi(|u|) + (1+\delta){\Phi_{\ast}}(|v|), \]
so, by $\delta\leqslant(\tilde{m}-1)/100\tilde{m}$, we conclude \eqref{eq:Xupperbound}.

\emph{Region $\mathcal{Y}_{\uparrow}$.}
Here we have, by $|u|<\Phi_{\ast}'(|v|)$, \eqref{eq:intPsiupper}, and \eqref{eq:quantities3},
\begin{align*}
\mathfrak{X}(u,v) & \leqslant \Phi(|u|) + \Big(1 + \delta \frac{\tilde{m}}{\tilde{m}-1} \frac{m}{m-1}\Big) {\Phi_{\ast}}(|v|) \\
& \leqslant \Big(1 + \frac{1}{100}\frac{m}{m-1}\Big) \big(\Phi(|u|) + {\Phi_{\ast}}(|v|)\big)
\end{align*}
and it remains to recall $m\geqslant2$; see \eqref{eq:onmMs}.
\end{proof}

We will also need certain derivative estimates for $\mathfrak{X}$.
Recall that, by writing $z=x+iy \in \mathbb{C}$, we can define operators of complex differentiation:
\[ \partial_z=\frac{\partial_x-i\partial_y}2, \quad
\partial_{\bar z}=\frac{\partial_x+i\partial_y}2. \]

\begin{lemma}\label{lm:Xderiv}
We have
\begin{align*}
|\partial_{\bar u}\mathfrak{X}(u,v)| & \leqslant \max\{\Phi'(|u|),|v|\}, \\
|\partial_{\bar v}\mathfrak{X}(u,v)| & \leqslant \Phi_{\ast}'(|v|)
\end{align*}
for any $(u,v)\in\mathbb{C}^2$.
\end{lemma}

\begin{proof}
It is sufficient to verify these estimates in $(\mathbb{C}\setminus\{0\})^2 \setminus\mathcal{Y}$, since continuity of partial derivatives will extend them to the whole domain $\mathbb{C}^2$.
Denote $u=u_r+iu_i$, $v=v_r+iv_i$.
Using the notation from the proof of Lemma~\ref{lm:XisC1} we can write
\begin{align*}
\partial_{\bar u}\mathfrak{X}(u,v)
& = \frac{1}{2}\big(\partial_{u_r}\mathfrak{X}(u_r+iu_i,v_r+iv_i)+ i\partial_{u_i}\mathfrak{X}(u_r+iu_i,v_r+iv_i)\big) \\
& = \frac{1}{2}\Big( \partial_1\mathfrak{B}(|u|,|v|)\frac{u_r}{|u|}
+ i\partial_1\mathfrak{B}(|u|,|v|)\frac{u_i}{|u|} \Big)
= \frac{\partial_1\mathfrak{B}(|u|,|v|) \,u}{2|u|}
\end{align*}
and, analogously,
\[ \partial_{\bar v}\mathfrak{X}(u,v) = \frac{\partial_2\mathfrak{B}(|u|,|v|) \,v}{2|v|}. \]

\emph{Region $\mathcal{Y}_{\downarrow}$.}
We have computed the partial derivatives of $\mathfrak{B}$ in \eqref{eq:parB1} and \eqref{eq:parB2}. Thanks to \eqref{eq:intPhiupper} and $\delta\leqslant(\tilde{m}-1)/100\tilde{m}$ we have
\[ |\partial_{\bar u}\mathfrak{X}(u,v)|
\leqslant \Big(\frac{1}{2}+\delta + \frac{\delta}{\tilde{m}-1}\Big) \Phi'(|u|)
\leqslant \Phi'(|u|). \]
Obviously, also
\[ |\partial_{\bar v}\mathfrak{X}(u,v)| \leqslant \Phi_{\ast}'(|v|). \]

\emph{Region $\mathcal{Y}_{\uparrow}$.}
This time, because of \eqref{eq:intPsiupper} and $|u|<\Phi_{\ast}'(|v|)$,
\[ |\partial_{\bar u}\mathfrak{X}(u,v)|
\leqslant \frac{1}{2}\Phi'(|u|) + \delta \frac{\tilde{m}}{\tilde{m}-1} |v| \leqslant \max\{\Phi'(|u|), |v|\} \]
and
\[ |\partial_{\bar v}\mathfrak{X}(u,v)|
\leqslant \Big(\frac{1}{2} + \frac{\delta}{2}\Big) \Phi_{\ast}'(|v|)
\leqslant \Phi_{\ast}'(|v|). \qedhere \]
\end{proof}

Let us finalize this section with a lower bound on the generalized Hessian of $\mathfrak{X}$.
It can also be thought of as a certain generalized convexity property of $\mathfrak{X}$.
Carbonaro and Dragi\v{c}evi\'{c} \cite{CarbonaroDragicevic20} pointed out that the generalized convexity of their Bellman function was essentially inherited from the generalized convexity of the power functions $u\mapsto|u|^p$ and $v\mapsto|v|^q$, which is, in turn, guaranteed by the $p$-ellipticity conditions $\Delta_p(A)>0$ and $\Delta_p(B)>0$.
Namely, the Bellman function was forced to contain mixed terms too, but a small weight was placed on them, not to ruin convexity properties of the whole expression.
The same philosophy applies here and the desired convexity properties of $\mathfrak{X}$ are essentially inherited from those of its ``main'' terms $u\mapsto\Phi(|u|)$ and $v\mapsto\Phi_{\ast}(|v|)$; also see Remark~\ref{rem:Bellmannmotivation} below.

\begin{lemma}
\label{lm:Xlower}
We have
\begin{equation}\label{eq:Xlowerbound}
H^{A(x),B(x)}_{\mathfrak{X}}[(u,v);(\zeta,\eta)] \geqslant \frac{1}{10} \Big(\frac{\tilde{M}}{\tilde{m}}\frac{\tilde{m}-1}{\tilde{M}-1}\Big)^{1/2} C_{p}(A,B)^{-1} |\zeta| |\eta|
\end{equation}
for $x\in\mathbb{R}^d$, $(u,v)\in(\mathbb{C}\setminus\{0\})^2 \setminus\mathcal{Y}$, and $(\zeta,\eta)\in(\mathbb{C}^d)^2$.
\end{lemma}

\begin{proof}
Using substitutions \eqref{eq:replacements} the lower bound \eqref{eq:Xlowerbound} can be rewritten as
\begin{equation}\label{eq:Xlower2}
\widetilde{H}^{A(x),B(x)}_{\mathfrak{X}}[(u,v);(\zeta,\eta)] \geqslant \frac{1}{10} \Big(\frac{\tilde{M}}{\tilde{m}}\frac{\tilde{m}-1}{\tilde{M}-1}\Big)^{1/2} C_{p}(A,B)^{-1} |\zeta| |\eta|
\end{equation}
and it will be verified separately in the two regions.

\emph{Region $\mathcal{Y}_{\downarrow}$.}
In this region Lemma~\ref{lm:hessians} can be applied with
\[ P(t) = (1+\delta) \Phi(t) + \delta t^2 \int_0^t \frac{\Phi'(s)\,\textup{d}s}{s^2},\quad
Q(t) = (1+\delta) {\Phi_{\ast}}(t), \]
noting that
\begin{align*}
\frac{1}{2}\Big(P''(t)+\frac{P'(t)}{t}\Big) & = \Big(\frac{1}{2}+\delta\Big) \Phi''(t) + \Big(\frac{1}{2}+2\delta\Big) \frac{\Phi'(t)}{t} + 2\delta \int_0^t \frac{\Phi'(s)\,\textup{d}s}{s^2}, \\
\frac{1}{2}\Big(P''(t)-\frac{P'(t)}{t}\Big) & = \Big(\frac{1}{2}+\delta\Big) \Phi''(t) - \frac{1}{2} \frac{\Phi'(t)}{t}.
\end{align*}
That way we end up with
{\allowdisplaybreaks
\begin{align*}
& \widetilde{H}^{A(x),B(x)}_{\mathfrak{X}}[(u,v);(\zeta,\eta)] \\
& = \mathop{\textup{Re}}\Bigg\langle A(x)\zeta, \ \frac{1}{2}\Big(\Phi''(|u|)+\frac{\Phi'(|u|)}{|u|}\Big) \zeta + \frac{1}{2}\Big(\Phi''(|u|)-\frac{\Phi'(|u|)}{|u|}\Big) \overline{\zeta} \Bigg\rangle_{\mathbb{C}^d} \\
& \quad + (1+\delta) \mathop{\textup{Re}}\Bigg\langle B(x)\eta, \ \frac{1}{2}\Big(\Phi_{\ast}''(|v|)+\frac{\Phi_{\ast}'(|v|)}{|v|}\Big) \eta + \frac{1}{2}\Big(\Phi_{\ast}''(|v|)-\frac{\Phi_{\ast}'(|v|)}{|v|}\Big) \overline{\eta} \Bigg\rangle_{\mathbb{C}^d} \\
& \quad + 2\delta \Phi''(|u|) \mathop{\textup{Re}}\big\langle A(x)\zeta, \mathop{\textup{Re}}\zeta \big\rangle_{\mathbb{C}^d}
+ 2\delta \Big(\frac{\Phi'(|u|)}{|u|} + \int_0^{|u|} \frac{\Phi'(s)\,\textup{d}s}{s^2}\Big) \underbrace{\mathop{\textup{Re}}\big\langle A(x)\zeta, \zeta \big\rangle_{\mathbb{C}^d}}_{\geqslant0}.
\end{align*}
}
Definition \eqref{eq:condPhi1}, positivity of $\Phi',\Phi'',\Phi_{\ast}',\Phi_{\ast}''$, and the choice of $\delta$ estimate this from below as
{\allowdisplaybreaks
\begin{align*}
& \widetilde{H}^{A(x),B(x)}_{\mathfrak{X}}[(u,v);(\zeta,\eta)] \\
& \geqslant \frac{1}{2} \Delta_{\Phi}(A) \Big(\Phi''(|u|)+\frac{\Phi'(|u|)}{|u|}\Big)  |\zeta|^2 + \frac{1+\delta}{2} \Delta_{{\Phi_{\ast}}}(B) \Big(\Phi_{\ast}''(|v|)+\frac{\Phi_{\ast}'(|v|)}{|v|}\Big) |\eta|^2
- 2\delta \Lambda(A) \Phi''(|u|) |\zeta|^2 \\
& \geqslant \frac{1}{4}\Delta_{p}(A) \Phi''(|u|) |\zeta|^2
+ \frac{1}{2}\Delta_{p}(B) \Phi_{\ast}''(|v|) |\eta|^2 \\
& \geqslant \frac{1}{2} \big(\Delta_{p}(A)\Delta_{p}(B)\big)^{1/2} \big(\Phi''(|u|)\Phi_{\ast}''(|v|)\big)^{1/2} |\zeta| |\eta| \\
& \geqslant \frac{1}{2} C_{p}(A,B)^{-1} \big(\Phi''(|u|)\Phi_{\ast}''(|v|)\big)^{1/2} |\zeta| |\eta|.
\end{align*}
}
Recall that in $\mathcal{Y}_{\downarrow}$ we have $|u|>\Phi_{\ast}'(|v|)$.
For the proof of \eqref{eq:Xlower2} it remains to observe that, since $\Phi''$ is increasing by \eqref{eq:Phicond1},
\[ \Phi''(|u|)\Phi_{\ast}''(|v|) \geqslant \Phi''(\Phi_{\ast}'(|v|)) \Phi_{\ast}''(|v|) = (\Phi'\circ\Phi_{\ast}')'(|v|) = 1. \]

\emph{Region $\mathcal{Y}_{\uparrow}$.}
In this region Lemma~\ref{lm:hessians} applies with
\[ P(t) = \Phi(t),\quad
Q(t) = {\Phi_{\ast}}(t),\quad
R(t) = \delta \int_0^t \frac{\textup{d}s}{\Phi_{\ast}'(s)}. \]
Taking into account
\[ R'(t) = \frac{\delta}{\Phi_{\ast}'(t)},\quad R''(t) = \frac{-\delta\Phi_{\ast}''(t)}{\Phi_{\ast}'(t)^2} \]
that lemma gives
{\allowdisplaybreaks
\begin{align*}
& \widetilde{H}^{A(x),B(x)}_{\mathfrak{X}}[(u,v);(\zeta,\eta)] \\
& = \mathop{\textup{Re}}\Bigg\langle A(x)\zeta, \ \frac{1}{2}\Big(\Phi''(|u|)+\frac{\Phi'(|u|)}{|u|}\Big) \zeta + \frac{1}{2}\Big(\Phi''(|u|)-\frac{\Phi'(|u|)}{|u|}\Big) \overline{\zeta} \Bigg\rangle_{\mathbb{C}^d} \\
& \quad + \mathop{\textup{Re}}\Bigg\langle B(x)\eta, \ \frac{1}{2}\Big(\Phi_{\ast}''(|v|)+\frac{\Phi_{\ast}'(|v|)}{|v|}\Big) \eta + \frac{1}{2}\Big(\Phi_{\ast}''(|v|)-\frac{\Phi_{\ast}'(|v|)}{|v|}\Big) \overline{\eta} \Bigg\rangle_{\mathbb{C}^d} \\
& \quad + 2\delta \Big(\int_0^{|v|} \frac{\textup{d}s}{\Phi_{\ast}'(s)}\Big) \mathop{\textup{Re}}\big\langle A(x)\zeta, \zeta \big\rangle_{\mathbb{C}^d} \\
& \quad + 2\delta \frac{|u|}{\Phi_{\ast}'(|v|)} \mathop{\textup{Re}}\big\langle A(x)\zeta, \mathop{\textup{Re}}\eta  \big\rangle_{\mathbb{C}^d}
+ 2\delta \frac{|u|}{\Phi_{\ast}'(|v|)} \mathop{\textup{Re}}\big\langle B(x)\eta, \mathop{\textup{Re}}\zeta \big\rangle_{\mathbb{C}^d} \\
& \quad - \delta \frac{|u|^2 \Phi_{\ast}''(|v|)}{\Phi_{\ast}'(|v|)^2} \mathop{\textup{Re}}\big\langle B(x)\eta, \mathop{\textup{Re}}\eta \big\rangle_{\mathbb{C}^d}
+ \delta \frac{|u|^2}{|v|\Phi_{\ast}'(|v|)} \mathop{\textup{Re}}\big\langle B(x)\eta, i\mathop{\textup{Im}}\eta \big\rangle_{\mathbb{C}^d}.
\end{align*}
}
Noting $\Phi',\Phi'',\Phi_{\ast}',\Phi_{\ast}''>0$ we see that the last expression is at least
{\allowdisplaybreaks
\begin{align*}
& \frac{1}{2} \Delta_{\Phi}(A) \Big( \Phi''(|u|) + \frac{\Phi'(|u|)}{|u|} \Big) |\zeta|^2
+ \frac{1}{2} \Delta_{{\Phi_{\ast}}}(B) \Big( \Phi_{\ast}''(|v|) + \frac{\Phi_{\ast}'(|v|)}{|v|} \Big) |\eta|^2 \\
& + 2\delta \lambda(A) \Big(\int_0^{|v|} \frac{\textup{d}s}{\Phi_{\ast}'(s)}\Big) |\zeta|^2
- 2\delta \big(\Lambda(A)+\Lambda(B)\big) \frac{|u|}{\Phi_{\ast}'(|v|)} |\zeta| |\eta| \\
& - \delta \Lambda(B) \Big(\frac{|u|^2\Phi_{\ast}''(|v|)}{\Phi_{\ast}'(|v|)^2} + \frac{|u|^2}{|v|\Phi_{\ast}'(|v|)}\Big) |\eta|^2.
\end{align*}
}
We can disregard the first term as nonnegative.
Since in $\mathcal{Y}_{\uparrow}$ we have $|u|<\Phi_{\ast}'(|v|)$, this whole expression is, in turn, bounded from below by
\[ 2\delta \lambda(A) \Big(\int_0^{|v|} \frac{\textup{d}s}{\Phi_{\ast}'(s)}\Big) |\zeta|^2
+ \Big(\frac{1}{2}\Delta_{p}(B) - \delta\Lambda(B)\Big) \Big( \Phi_{\ast}''(|v|) + \frac{\Phi_{\ast}'(|v|)}{|v|} \Big) |\eta|^2
- 2\delta \big(\Lambda(A)+\Lambda(B)\big) |\zeta| |\eta|. \]
The last display can be viewed as a quadratic form in $|\zeta|$ and $|\eta|$, and it is at least
\begin{align*}
& \Big(2\delta\lambda(A)\Delta_{p}(B) \frac{\Phi_{\ast}'(|v|)}{|v|} \int_0^{|v|} \frac{\textup{d}s}{\Phi_{\ast}'(s)}\Big)^{1/2} |\zeta| |\eta| - 2\delta \big(\Lambda(A)+\Lambda(B)\big) |\zeta| |\eta| \\
& \geqslant \Big(\delta\lambda(A)\Delta_{p}(B) \frac{\tilde{M}}{\tilde{M}-1}\Big)^{1/2} |\zeta| |\eta|
\geqslant \frac{1}{10} \Big(\frac{\tilde{m}-1}{\tilde{m}}\frac{\tilde{M}}{\tilde{M}-1}\Big)^{1/2} C_{p}(A,B)^{-1} |\zeta| |\eta|,
\end{align*}
where we also used \eqref{eq:intPsiupper} and the fact that $\delta$ was given by \eqref{eq:choiceofdelta}.
This proves \eqref{eq:Xlower2} again.
\end{proof}

\begin{remark}\label{rem:Bellmannmotivation}
Let us finalize this section with a short reasoning that can lead naturally to formula \eqref{eq:mainBellman} for the Bellman function $\mathfrak{X}(u,v)$.
First, the main terms $\Phi(|u|)$ and $\Phi_{\ast}(|v|)$ need to be included, as they are responsible for the generalized convexity, as we have already remarked.
In Lemma~\ref{lm:Xlower} we need a concrete quantitative lower bound on the generalized Hessian of $\mathfrak{X}$ and these two terms would be sufficient in the region $\mathcal{Y}_{\downarrow}$.
However, in the region $\mathcal{Y}_{\uparrow}$ we compensate convexity with a ``hidden term'' (in the language of \cite[Section~8.2]{NazarovTreil96}) of the form $\delta |u|^2 R(|v|)$ for a sufficiently small $\delta>0$.
In order to still have the upper bound of Lemma~\ref{lm:Xupper}, the best we can do (in complete generality) in the region $\mathcal{Y}_{\uparrow}$ is to estimate $|u|^2 R(|v|) \leq \Phi'_{\ast}(|v|)^2 R(|v|)$.
We would like this to be controlled by $\Phi_{\ast}(|v|)$, so it makes sense to take $R(|v|)$ to be roughly $\Phi_{\ast}(|v|)/\Phi'_{\ast}(|v|)^2$.
However, the last expression is not a particularly good choice, as its convexity would also need to be characterized by properties of the third derivative $\Phi'''_{\ast}$.
Luckily, the assumptions from Subsection~\ref{subsec:Orlicz} offer many possibilities for quantities that are comparable to $\Phi_{\ast}/(\Phi'_{\ast})^2$ and the most convenient choice is the primitive function of $1/\Phi'_{\ast}$, leading to
$R(|v|)=\int_{0}^{|v|}1/\Phi'_{\ast}$.
It is important that $\mathfrak{X}$ is $\textup{C}^1$, so its definitions on $\mathcal{Y}_{\downarrow}$ and $\mathcal{Y}_{\uparrow}$ need to match nicely. For this reason, computation \eqref{eq:integralcomput} then reveals us how to complete its formula on $\mathcal{Y}_{\downarrow}$.
Note that the term $\delta|u|^2\int_{0}^{|u|}\Phi'(s)\,\textup{d}s/s^2$ is particularly interesting, as it was not visible in the classical special case \eqref{eq:particularBellman}.
\end{remark}

\section{Proof of Theorem~\ref{thm:mainthm}}
\label{sec:proofofthm}
As discussed in Section \ref{subsec:formulationresults}, to prove Theorem~\ref{thm:mainthm} it is enough to establish the dehomogenized estimate \eqref{eq:bilinorl2} from Remark~\ref{rem:dehomogenization}. In its proof we will use the heat flow method and closely follow the outline by Carbonaro and Dragi\v{c}evi\'{c} \cite[Section 6]{CarbonaroDragicevic20} (also see \cite[Section 6]{CarbonaroDragicevicKovacSkreb21}). We will be very brief because what follows is a straightforward adaptation of their arguments. On the other hand, we still include a few details to indicate how certain formulas generalize from powers to Young functions $\Phi$ and ${\Phi_{\ast}}$.

\subsection{Regularization}
\label{subsec:regularization}
In the proof of Theorem~\ref{thm:mainthm} we will need a smoother version of the constructed Bellman function \eqref{eq:mainBellman}. To be more precise, we want to replace $\mathfrak{X}$ by a function that satisfies similar properties to those in Lemmas~\ref{lm:Xupper}--\ref{lm:Xlower} but is, in addition, also of class $\textup{C}^\infty$ everywhere on $\mathbb{C}^2$, and not only in $(\mathbb{C}\setminus\{0\})^2\setminus\mathcal{Y}$. Mollification of the Bellman function for this purpose has already been employed in \cite{PetermichlVolberg02} and \cite{NazarovVolberg03}. In a similar context, this ``regularization'' was performed in almost exactly the same way in \cite[Subsection~5.1]{CarbonaroDragicevic20}.

Let us fix a nonnegative radial $\textup{C}^\infty$ function $\varphi$ on $\mathbb{C}^2\equiv\mathbb{R}^4$, supported in the standard unit ball, and such that $\int_{\mathbb{C}^2}\varphi=1$. For a given $\nu\in(0,1]$ and $(w,z)\in\mathbb{C}^2$ we define $\varphi_\nu(w,z) := \nu^{-4} \varphi(w/\nu,z/\nu)$. Note that $\varphi_\nu$ are $\textup{L}^1$-normalized dilates of $\varphi$. Consider
\[ \mathfrak{X}_{\nu}:=\mathfrak{X}\ast\varphi_\nu, \]
i.e.,
\[ \mathfrak{X}_{\nu}(u,v) :=\int_{\mathbb{C}^2} \mathfrak{X}(u-w,v-z) \varphi_{\nu}(w,z) \,\textup{d}w \,\textup{d}z \]
for $(u,v)\in\mathbb{C}^2$.
Here $\textup{d}w$ and $\textup{d}z$ denote integration with respect to the two-dimensional Lebesgue measure on $\mathbb{C}\equiv\mathbb{R}^2$; it should not be confused with complex integration.
Clearly, $\mathfrak{X}_\nu$ is of class $\textup{C}^\infty$ on the whole $\mathbb{C}^2$, since it is a convolution of $\mathfrak{X}$ with a smooth function.
We still consider fixed $\Phi,{\Phi_{\ast}},A,B$ that satisfy hypotheses of Theorem~\ref{thm:mainthm}.

\begin{proposition}
\label{prop:Xmollified}
\begin{enumerate}[(a)]
\item \label{eq:Xmoll1}
For $\nu\in(0,1]$ and $(u,v)\in\mathbb{C}^2$ we have:
\[ 0 \leqslant \mathfrak{X}_\nu(u,v) \leqslant 2\max\Big\{1,\frac{M}{\tilde{m}}\Big\}\big(\Phi(|u|+\nu)+ {\Phi_{\ast}}(|v|+\nu)\big). \]
\item \label{eq:Xmoll2}
For $\nu\in(0,1]$ and $(u,v)\in\mathbb{C}^2$ we have:
\begin{align*}
\big|\partial_{\bar{u}}\mathfrak{X}_{\nu}(u,v)\big| & \leqslant \max\big\{\Phi'(|u|+\nu),|v|+\nu\big\}, \\
\big|\partial_{\bar{v}}\mathfrak{X}_{\nu}(u,v)\big| & \leqslant \Phi_{\ast}'(|v|+\nu).
\end{align*}
\item \label{eq:Xmoll3}
For $\nu\in(0,1]$, $x\in\mathbb{R}^d$, $(u,v)\in\mathbb{C}^2$, and $(\zeta,\eta)\in(\mathbb{C}^d)^2$ we have:
\[ H^{A(x),B(x)}_{\mathfrak{X}_\nu}[(u,v);(\zeta,\eta)] \geqslant \frac{1}{10} \Big(\frac{\tilde{M}}{\tilde{m}}\frac{\tilde{m}-1}{\tilde{M}-1}\Big)^{1/2} C_{p}(A,B)^{-1} |\zeta| |\eta|. \]
\end{enumerate}
\end{proposition}

\begin{proof}
\emph{Estimate \eqref{eq:Xmoll1}.}
By the definition of $\mathfrak{X}_\nu$ and estimate \eqref{eq:Xupperbound} from Lemma~\ref{lm:Xupper} we easily get
\begin{align*}
\mathfrak{X}_\nu(u,v)
& \leqslant 2\max\Big\{1,\frac{M}{\tilde{m}}\Big\}\int_{\mathbb{C}^2} \big(\Phi(|u-w|)+{\Phi_{\ast}}(|v-z|)\big) \varphi_\nu(w,z) \, \textup{d}w \,\textup{d}z \\
& \leqslant 2\max\Big\{1,\frac{M}{\tilde{m}}\Big\} \big(\Phi(|u|+\nu)+{\Phi_{\ast}}(|v|+\nu)\big).
\end{align*}
Here, in the last inequality, we used $\int_{\mathbb{C}^2} \varphi_\nu=1$ and that for $(w,z)$ in the support of $\varphi_\nu$ we have $|u-w|\leqslant |u|+\nu$ and $|v-z|\leqslant |v|+\nu$, while $\Phi$ and ${\Phi_{\ast}}$ are increasing.

\emph{Estimates \eqref{eq:Xmoll2}.}
Recall that Lemma~\ref{lm:XisC1} guarantees that the first-order partial derivatives of $\mathfrak{X}$ are continuous.
By Lemma~\ref{lm:Xderiv} we have
\begin{align*}
\big|\partial_{\bar{u}}\mathfrak{X}_{\nu}(u,v)\big|
& = \Big|\int_{\mathbb{C}^2}\partial_{\bar{u}}\mathfrak{X}(u-w,v-z)\varphi_\nu(w,z) \,\textup{d}w \,\textup{d}z\Big| \\
& \leqslant \int_{\mathbb{C}^2}\max\big\{\Phi'(|u-w|), |v-z|\big\} \,\varphi_\nu(w,z) \,\textup{d}w \,\textup{d}z \\
& \leqslant \max\left\{\Phi'(|u|+\nu), |v|+\nu\right\}
\end{align*}
and analogously
\[ \big|\partial_{\bar{v}}\mathfrak{X}_{\nu}(u,v)\big| \leqslant \Phi_{\ast}'(|v|+\nu). \]
Here we used that $\Phi'$ and $\Phi_{\ast}'$ are increasing too; recall \eqref{eq:bijections}.

\emph{Estimate \eqref{eq:Xmoll3}.}
Lemma~\ref{lm:XisC1} also guarantees that the second-order derivatives of $\mathfrak{X}$ are locally integrable functions on $\mathbb{R}^4$ defined on the complement of the critical surface and coordinate hyperplanes.
Combining classical results \cite[Theorem~6.3.11]{Cohn13} and \cite[Theorem~2.1]{Hebey99} we see that the second-order partial derivatives of $\mathfrak{X}$ can equally well be computed in the weak (i.e., distributional) sense. In particular, the generalized Hessian of $\mathfrak{X}\ast\varphi_{\nu}$ is the convolution of the generalized Hessian of $\mathfrak{X}$ and $\varphi_{\nu}$. The former exists almost everywhere and satisfies the bound \eqref{eq:Xlowerbound} from Lemma~\ref{lm:Xlower} at those points.
Therefore, we still have
\begin{align*}
H^{A(x),B(x)}_{\mathfrak{X}_\nu}[(u,v);(\zeta,\eta)]
& = \int_{\mathbb{C}^2} H^{A(x),B(x)}_{\mathfrak{X}}[(u-w,v-z);(\zeta,\eta)] \,\varphi_\nu(w,z) \,\textup{d}w \,\textup{d}z \\
& \geqslant \frac{1}{10} \Big(\frac{\tilde{M}}{\tilde{m}}\frac{\tilde{m}-1}{\tilde{M}-1}\Big)^{1/2} C_{p}(A,B)^{-1} |\zeta| |\eta|. \qedhere
\end{align*}
\end{proof}

\subsection{Proof for smooth matrix functions}
First, we assume that the entries of $A$ and $B$ are bounded $\textup{C}^1$ functions that also have bounded derivatives. In addition to the previously fixed $\Phi,{\Phi_{\ast}},A,B$, now we also take $f,g\in\textup{C}_{c}^{\infty}(\mathbb{R}^d)$. Let us choose any radial $\textup{C}^\infty$ function $\psi\colon\mathbb{R}^d\to[0,1]$ that is constantly $1$ on the standard unit ball, while vanishing on its double dilate around the origin.
This time we normalize dilates of $\psi$ in $\textup{L}^\infty$ norm and write $\psi_R(x) := \psi(x/R)$ for any $R\in(0,\infty)$ and $x\in\mathbb{R}^d$.
Finally, for each $\nu\in(0,1]$ recall the mollified Bellman function $\mathfrak{X}_\nu$ from Subsection~\ref{subsec:regularization}.

Just as in \cite{CarbonaroDragicevic20}, for given $R\in(0,\infty)$ and $\nu\in(0,1]$ we define $\mathcal{E}_{R,\nu}\colon[0,\infty)\to [0,\infty)$ as
\[ \mathcal{E}_{R,\nu}(t) := \int_{\mathbb{R}^{d}} \psi_R(x) \,\mathfrak{X}_{\nu}\big((T_{t}^{A}f)(x),(T_{t}^{B}g)(x)\big) \,\textup{d}x. \]
The following manipulations were justified in \cite[Section 3.1]{CarbonaroDragicevic20} and \cite[Section 4.1]{CarbonaroDragicevic20}.
Proposition~\ref{prop:Xmollified}\,(\ref{eq:Xmoll1}) gives an upper bound on the following integral for a fixed time $T\in(0,\infty)$:
{\allowdisplaybreaks
\begin{align*}
-\int_{0}^{T} \mathcal{E}_{R,\nu}'(t) \,\textup{d}t
& = \mathcal{E}_{R,\nu}(0) - \mathcal{E}_{R,\nu}(T)
\leqslant \mathcal{E}_{R,\nu}(0) \\
& =\int_{\mathbb{R}^d} \psi_R(x) \,\mathfrak{X}_{\nu}\big(f(x),g(x)\big) \,\textup{d}x \\
& \leqslant 2\max\Big\{1,\frac{M}{\tilde{m}}\Big\}\int_{\mathbb{R}^d} \psi_R(x) \,\big( \Phi(|f(x)|+\nu)+{\Phi_{\ast}}(|g(x)|+\nu) \big) \,\textup{d}x.
\end{align*}
}
On the other hand, \cite[Proposition 4.3]{CarbonaroDragicevic20} and Proposition~\ref{prop:Xmollified}\,(\ref{eq:Xmoll3}) give a lower bound on the same integral:
{\allowdisplaybreaks
\begin{align*}
& -\int_{0}^{T} \mathcal{E}_{R,\nu}'(t) \,\textup{d}t
= -\int_{0}^{T} \int_{\mathbb{R}^{d}} \psi_R(x) \,\frac{\partial}{\partial t}\mathfrak{X}_{\nu}\big((T_{t}^{A}f)(x),(T_{t}^{B}g)(x)\big) \,\textup{d}x \,\textup{d}t \\
& = \int_{0}^{T} \int_{\mathbb{R}^d} \psi_R(x) \, H^{A(x),B(x)}_{\mathfrak{X}_\nu}\big[\big((T_{t}^{A}f)(x),(T_{t}^{B}g)(x)\big);\big((\nabla T_{t}^{A}f)(x),(\nabla T_{t}^{B}g)(x)\big)\big] \,\textup{d}x \,\textup{d}t + \mathcal{R}_{T,R,\nu} \\
& \geqslant \frac{1}{10} \Big(\frac{\tilde{M}}{\tilde{m}}\frac{\tilde{m}-1}{\tilde{M}-1}\Big)^{1/2} C_{p}(A,B)^{-1} \int_{0}^{T} \int_{\mathbb{R}^d} \psi_R(x) \,\big|(\nabla T_{t}^{A}f)(x)\big| \,\big|(\nabla T_{t}^{B}g)(x)\big| \,\textup{d}x \,\textup{d}t + \mathcal{R}_{T,R,\nu},
\end{align*}
}
where $\mathcal{R}_{T,R,\nu}$ is the remainder (a.k.a.\@ the \emph{error-term} in \cite{CarbonaroDragicevic20}), given as
\begin{align*}
\mathcal{R}_{T,R,\nu} := 2\mathop{\textup{Re}}
\int_{0}^{T} \int_{\mathbb{R}^d}
\Big( & (\partial_{\bar{u}} \mathfrak{X}_{\nu})\big((T_{t}^{A}f)(x),(T_{t}^{B}g)(x)\big) \,\big\langle(\nabla\psi_R)(x), A(x)(\nabla T_{t}^{A} f)(x) \big\rangle_{\mathbb{C}^{d}} \\
& + (\partial_{\bar{v}} \mathfrak{X}_{\nu})\big((T_{t}^{A}f)(x),(T_{t}^{B}g)(x)\big) \,\big\langle (\nabla\psi_R)(x), B(x)(\nabla T_{t}^{B} g)(x) \big\rangle_{\mathbb{C}^{d}}\Big) \,\textup{d}x \,\textup{d}t.
\end{align*}
Proposition~\ref{prop:Xmollified}\,(\ref{eq:Xmoll2}) controls this remainder as
{\allowdisplaybreaks
\begin{align*}
|\mathcal{R}_{T,R,\nu}|
& \leqslant 2 \Lambda(A) \int_{0}^{T} \int_{\mathbb{R}^d}|\nabla\psi_R(x)| \,\Phi'\big(|(T_t^A f)(x)|+\nu\big) \,|(\nabla T_t^A f)(x)| \,\textup{d}x \,\textup{d}t \\
& + 2 \Lambda(A) \int_{0}^{T} \int_{\mathbb{R}^d}|\nabla\psi_R(x)| \,\big(|(T_t^B g)(x)|+\nu\big) \,|(\nabla T_t^A f)(x)| \,\textup{d}x \,\textup{d}t \\
& + 2 \Lambda(B) \int_{0}^{T} \int_{\mathbb{R}^d}|\nabla\psi_R(x)| \,\Phi_{\ast}'\big(|(T_t^B g)(x)|+\nu\big) \,|(\nabla T_t^B g)(x)| \,\textup{d}x \,\textup{d}t.
\end{align*}
}
By reasoning as in the proof of \cite[Lemma 6.1]{CarbonaroDragicevic20}, semigroup $\textup{L}^\infty$ estimates and Davies-Gaffney-type estimates now easily show
\[ \limsup_{R\to\infty} \limsup_{\nu\to0^+} |\mathcal{R}_{T,R,\nu}| = 0 \]
for any fixed $T\in(0,\infty)$.
Thus, we first let $\nu\to0^+$ and then send $R\to\infty$, both in the upper estimate and in the lower estimate above.
Combining the two immediately gives us
\begin{align*}
\frac{1}{10} \Big(\frac{\tilde{M}}{\tilde{m}}\frac{\tilde{m}-1}{\tilde{M}-1}\Big)^{1/2} C_{p}(A,B)^{-1} \int_{0}^{T} \int_{\mathbb{R}^d} \big|(\nabla T_{t}^{A}f)(x)\big| \,\big|(\nabla T_{t}^{B}g)(x)\big| \,\textup{d}x \,\textup{d}t & \\
\leqslant 2\max\Big\{1,\frac{M}{\tilde{m}}\Big\}\int_{\mathbb{R}^d} \big( \Phi(|f(x)|)+{\Phi_{\ast}}(|g(x)|) \big) \,\textup{d}x & .
\end{align*}
In the limit as $T\to\infty$ we obtain precisely \eqref{eq:bilinorl2}.

\subsection{Proof for non-smooth matrix functions}
Extension of the estimate \eqref{eq:bilinorl2} to arbitrary $A$ and $B$ is performed exactly as in \cite{CarbonaroDragicevic20}. In \cite[Appendix]{CarbonaroDragicevic20} the authors define smooth approximations $A_\varepsilon$ and $B_\varepsilon$ such that $\nabla T_t^{A_\varepsilon}f \to \nabla T_t^{A}f$ in the $\textup{L}^2$ norm as $\varepsilon\to0^+$, $\lambda(A)\leqslant\lambda(A_\varepsilon)\leqslant\Lambda(A_\varepsilon)\leqslant\Lambda(A)$, $\Delta_p(A_\varepsilon)\geqslant\Delta_p(A)$, etc. The proof is then finalized as in \cite[Section~6]{CarbonaroDragicevic20}, by applying the previously established smooth case and letting $\varepsilon\to0^+$.


\section*{Acknowledgements}
This work was supported in part by the \emph{Croatian Science Foundation} project UIP-2017-05-4129 (MUNHANAP).
The authors are grateful to Oliver Dragi\v{c}evi\'{c} for introducing them to elliptic operators with non-smooth complex coefficients during his stay at the University of Zagreb in the Spring of 2019.
The authors would also like to thank the anonymous referee for numerous excellent remarks that have greatly improved readability of the text.


\bibliography{bilinear_div_Orlicz}{}

\begin{thebibliography}{10}

\bibitem{CarbonaroDragicevic13}
Andrea Carbonaro and Oliver Dragi\v{c}evi\'c.
\newblock {B}ellman function and dimension-free estimates in a theorem of
  {B}akry.
\newblock {\em J. Funct. Anal.}, 265:1085--1104, 2013.

\bibitem{CarbonaroDragicevic17}
Andrea Carbonaro and Oliver Dragi\v{c}evi\'c.
\newblock Functional calculus for generators of symmetric contraction
  semigroups.
\newblock {\em Duke Math. J.}, 166(5):937--974, 2017.

\bibitem{CarbonaroDragicevic19b}
Andrea Carbonaro and Oliver Dragi\v{c}evi\'c.
\newblock Bilinear embedding for {S}chr\"odinger-type operators with complex
  coefficients.
\newblock https://arxiv.org/abs/1908.00143, 2019.

\bibitem{CarbonaroDragicevic19}
Andrea Carbonaro and Oliver Dragi\v{c}evi\'c.
\newblock Bounded holomorphic functional calculus for nonsymmetric
  {O}rnstein-{U}hlenbeck operators.
\newblock {\em Ann. Sc. Norm. Super. Pisa Cl. Sci. (5)}, 19(4):1497--1533,
  2019.

\bibitem{CarbonaroDragicevic20b}
Andrea Carbonaro and Oliver Dragi\v{c}evi\'c.
\newblock Bilinear embedding for divergence-form operators with complex
  coefficients on irregular domains.
\newblock {\em Calc. Var. Partial Differential Equations}, 59(3):36 pp., 2020.
\newblock Paper No. 104.

\bibitem{CarbonaroDragicevic20}
Andrea Carbonaro and Oliver Dragi\v{c}evi\'{c}.
\newblock Convexity of power functions and bilinear embedding for
  divergence-form operators with complex coefficients.
\newblock {\em J. Eur. Math. Soc. (JEMS)}, 22(10):3175--3221, 2020.

\bibitem{CarbonaroDragicevicKovacSkreb21}
Andrea Carbonaro, Oliver Dragi\v{c}evi\'c, Vjekoslav Kova\v{c}, and
  Kristina~Ana \v{S}kreb.
\newblock Trilinear embedding for divergence-form operators with complex
  coefficients.
\newblock https://arxiv.org/abs/2101.11694, 2021.

\bibitem{CialdeaMazya05}
Alberto Cialdea and Vladimir Maz'ya.
\newblock Criterion for the {$L^p$}-dissipativity of second order differential
  operators with complex coefficients.
\newblock {\em J. Math. Pures Appl. (9)}, 84(8):1067--1100, 2005.

\bibitem{CialdeaMazya21}
Alberto Cialdea and Vladimir Maz'ya.
\newblock Criterion for the functional dissipativity of second order
  differential operators with complex coefficients.
\newblock {\em Nonlinear Anal.}, 206:Paper No. 112215, 26 pp., 2021.

\bibitem{CialdeaMazya21b}
Alberto Cialdea and Vladimir Maz'ya.
\newblock Criterion for the functional dissipativity of the {L}am\'{e}
  operator.
\newblock https://arxiv.org/abs/2108.06299, 2021.

\bibitem{Cianchi98}
Andrea Cianchi.
\newblock An optimal interpolation theorem of {M}arcinkiewicz type in {O}rlicz
  spaces.
\newblock {\em J. Funct. Anal.}, 153(2):357--381, 1998.

\bibitem{Cohn13}
Donald~L. Cohn.
\newblock {\em Measure theory}.
\newblock Birkh\"{a}user Advanced Texts: Basler Lehrb\"{u}cher. [Birkh\"{a}user
  Advanced Texts: Basel Textbooks]. Birkh\"{a}user/Springer, New York, second
  edition, 2013.

\bibitem{DindosPipher19}
Martin Dindo\v{s} and Jill Pipher.
\newblock Regularity theory for solutions to second order elliptic operators
  with complex coefficients and the ${L}^{p}$ {D}irichlet problem.
\newblock {\em Adv. Math.}, 341:255--298, 2019.

\bibitem{DragicevicVolberg06}
Oliver Dragi{\v{c}}evi{\'c} and Alexander Volberg.
\newblock Bellman functions and dimensionless estimates of {L}ittlewood-{P}aley
  type.
\newblock {\em J. Operator Theory}, 56(1):167--198, 2006.

\bibitem{DragicevicVolberg11}
Oliver Dragi{\v{c}}evi{\'c} and Alexander Volberg.
\newblock Bilinear embedding for real elliptic differential operators in
  divergence form with potentials.
\newblock {\em J. Funct. Anal.}, 261(10):2816--2828, 2011.

\bibitem{DragicevicVolberg12}
Oliver Dragi{\v{c}}evi{\'c} and Alexander Volberg.
\newblock Linear dimension-free estimates in the embedding theorem for
  {S}chr\"odinger operators.
\newblock {\em J. Lond. Math. Soc. (2)}, 85(1):191--222, 2012.

\bibitem{DragicevicVolberg05}
Oliver Dragi\v{c}evi\'c and Alexander Volberg.
\newblock Bellman function, {L}ittlewood-{P}aley estimates and asymptotics for
  the {A}hlfors-{B}eurling operator in {$L^p(\mathbb{C})$}.
\newblock {\em Indiana Univ. Math. J.}, 54(4):971--995, 2005.

\bibitem{HardyLittlewoodPolya52}
Godfrey~Harold Hardy, John~Edensor Littlewood, and George P\'{o}lya.
\newblock {\em Inequalities}.
\newblock Cambridge University Press, 1952.

\bibitem{HarjulehtoHasto19}
Petteri Harjulehto and Peter H\"{a}st\"{o}.
\newblock {\em Orlicz spaces and generalized {O}rlicz spaces}, volume 2236 of
  {\em Lecture Notes in Mathematics}.
\newblock Springer, Cham, 2019.

\bibitem{Hebey99}
Emmanuel Hebey.
\newblock {\em Nonlinear analysis on manifolds: {S}obolev spaces and
  inequalities}, volume~5 of {\em Courant Lecture Notes in Mathematics}.
\newblock New York University, Courant Institute of Mathematical Sciences, New
  York; American Mathematical Society, Providence, 1999.

\bibitem{KarlovichMaligranda01}
Alexei~Yu. Karlovich and Lech Maligranda.
\newblock On the interpolation constant for {O}rlicz spaces.
\newblock {\em Proc. Amer. Math. Soc.}, 129(9):2727--2739, 2001.

\bibitem{Kreinetal82}
Selim~Grigor'evich Kre\u{\i}n, Yuri~Ivanovich Petun\={\i}n, and
  Evgeni\u{\i}~Mikha\u{\i}lovich Sem\"{e}nov.
\newblock {\em Interpolation of linear operators}, volume~54 of {\em
  Translations of Mathematical Monographs}.
\newblock American Mathematical Society, Providence, 1982.

\bibitem{NazarovTreil96}
Fedor Nazarov and Sergei Treil.
\newblock The hunt for a {B}ellman function: applications to estimates for
  singular integral operators and to other classical problems of harmonic
  analysis.
\newblock {\em Algebra i Analiz}, 8(5):32--162, 1996.

\bibitem{NazarovTreilVolberg99}
Fedor Nazarov, Sergei Treil, and Alexander Volberg.
\newblock The {B}ellman functions and two-weight inequalities for {H}aar
  multipliers.
\newblock {\em J. Amer. Math. Soc.}, 12(4):909--928, 1999.

\bibitem{NazarovTreilVolberg01}
Fedor Nazarov, Sergei Treil, and Alexander Volberg.
\newblock Bellman function in stochastic control and harmonic analysis.
\newblock In {\em Systems, approximation, singular integral operators, and
  related topics ({B}ordeaux, 2000)}, volume 129 of {\em Oper. Theory Adv.
  Appl.}, pages 393--423. Birkh\"{a}user, Basel, 2001.

\bibitem{NazarovVolberg03}
Fedor Nazarov and Alexander Volberg.
\newblock Heat extension of the {B}eurling operator and estimates for its norm.
\newblock {\em Algebra i Analiz (translation in St. Petersburg Math. Journal)},
  15(4):142--158 (563--573), 2003 (2004).

\bibitem{Nittka12}
Robin Nittka.
\newblock Projections onto convex sets and {$L^p$}-quasi-contractivity of
  semigroups.
\newblock {\em Arch. Math. (Basel)}, 98(4):341--353, 2012.

\bibitem{Ouhabaz05}
El~Maati Ouhabaz.
\newblock {\em Analysis of heat equations on domains}, volume~31 of {\em London
  Mathematical Society Monographs Series}.
\newblock Princeton University Press, Princeton, NJ, 2005.

\bibitem{PetermichlVolberg02}
Stefanie Petermichl and Alexander Volberg.
\newblock Heating of the {A}hlfors-{B}eurling operator: weakly quasiregular
  maps on the plane are quasiregular.
\newblock {\em Duke Math. J.}, 112(2):281--305, 2002.

\bibitem{RaoRen91}
Malempati~Madhusudana Rao and Zhong~Dao Ren.
\newblock {\em Theory of {O}rlicz spaces}, volume 146 of {\em Monographs and
  Textbooks in Pure and Applied Mathematics}.
\newblock Marcel Dekker, Inc., New York, 1991.

\bibitem{Torchinsky76}
Alberto Torchinsky.
\newblock Interpolation of operations and {O}rlicz classes.
\newblock {\em Studia Math.}, 59(2):177--207, 1976/77.

\bibitem{TreilVolberg16}
Sergei Treil and Alexander Volberg.
\newblock Entropy conditions in two weight inequalities for singular integral
  operators.
\newblock {\em Adv. Math.}, 301:499--548, 2016.

\bibitem{Zygmund56}
Antoni Zygmund.
\newblock On a theorem of {M}arcinkiewicz concerning interpolation of
  operations.
\newblock {\em J. Math. Pures Appl. (9)}, 35:223--248, 1956.

\end{thebibliography}
\bibliographystyle{plain}

\end{document}